\documentclass[a4paper,11 pt,twoside]{amsart}

\usepackage[utf8]{inputenc}
\usepackage[english]{babel}
\usepackage{times}
\usepackage{amsmath}
\usepackage{amsfonts}
\usepackage{amssymb}
\usepackage{amsmath}
\usepackage{amsthm}
\usepackage{graphicx}
\usepackage{array}
\usepackage{color}
\usepackage{mathrsfs}
\usepackage{hyperref}
\usepackage{esint} 
\usepackage{tikz}
\usepackage{ upgreek }
\usepackage{enumitem}

\definecolor{darkgreen}{rgb}{0.1,0.7,0.1}
\definecolor{darkred}{rgb}{0.7,0.1,0.1}

\DeclareMathOperator{\Tr}{Tr}

\newtheorem{theorem}{Theorem}
\newtheorem{lemma}{Lemma}[section]
\newtheorem{proposition}[lemma]{Proposition}

\newtheorem{remark}[lemma]{Remark}

\newcommand{\ru}{\hat{r}}

\newcommand{\yu}{\hat{y}}

\newcommand\symb[2][\bf]{{\mathchoice{\hbox{#1#2}}{\hbox{#1#2}}%
        {\hbox{\scriptsize#1#2}}{\hbox{\tiny#1#2}}}}

\def\R{{\symb R}}
\def\N{{\symb N}}
\def\Z{{\symb Z}}
\def\C{{\symb C}}

\def\P{{\symb P}}

\def\un{\mathbf{1}}

\renewcommand{\P}{\mathbb{P}}
\newcommand{\E}{\mathbb{E}}

\newcommand{\cB}{\mathcal{B}}
\newcommand{\cC}{\mathcal{C}}
\newcommand{\cD}{\mathcal{D}}
\newcommand{\cE}{\mathcal{E}}
\newcommand{\cF}{\mathcal{F}}

\newcommand{\cH}{\mathcal{H}}

\newcommand{\cL}{\mathcal{L}}
\newcommand{\cM}{\mathcal{M}}

\newcommand{\ccC}{\mathscr{C}}

\newcommand{\gl}{\lambda}

\hypersetup{
citecolor=blue,
colorlinks=true, 
breaklinks=true, 
urlcolor= blue, 
linkcolor= black, 
bookmarksopen=true, 
pdftitle={Anderson Hamiltonian}, 
}

\begin{document}

\title[Anderson localization]{Anderson localization for the $1$-d Schr\"odinger operator\\with white noise potential}

\author{Laure Dumaz}
\address{
	CNRS \& Department of Mathematics and Applications, \'Ecole Normale Sup\'erieure (Paris), 45 rue d’Ulm, 75005 Paris, France}
\email{laure.dumaz@ens.fr}

\author{Cyril Labb\'e}
\address{Universit\'e Paris Cit\'e, Laboratoire de Probabilit\'es, Statistiques et Mod\'elisation, UMR 8001, F-75205 Paris, France}
\email{clabbe@lpsm.paris}

\vspace{2mm}

\date{\today}

\maketitle

\begin{abstract}
We consider the random Schr\"odinger operator on $\R$ obtained by perturbing the Laplacian with a white noise. We prove that Anderson localization holds for this operator: almost surely the spectral measure is pure point and the eigenfunctions are exponentially localized. We give two separate proofs of this result. We also present a detailed construction of the operator and relate it to the parabolic Anderson model. Finally, we discuss the case where the noise is smoothed out.

\medskip

\end{abstract}

\setcounter{tocdepth}{1}
\tableofcontents

\section{Introduction}

The present article is concerned with the spectral properties of the random Schr\"odinger operator $\cH$ defined by
$$ \cH = -\partial_x^2 + \xi\;,\quad \mbox{ on }\R\;,$$
where $\xi$ is Gaussian white noise.\\

In a series of recent works~\cite{DL_Bottom,DL_Crossover,DL_Critical}, we considered the finite volume version of this operator, namely the operator $\cH_L$ defined by
$$ \cH_L = -\partial_x^2 + \xi\;,\quad \mbox{ on }(-L/2,L/2)\;,$$
(endowed with Dirichlet b.c.~but this is unimportant) and investigated the behavior of its eigenvalues and eigenfunctions in the infinite volume limit $L\to\infty$. A surprising diversity of behaviors arises according to the region of the spectrum one looks at. Let us simply mention that the transition between localized and delocalized eigenfunctions occurs at energies $E=E(L)$ of order $L$: at energies much smaller than $L$, most of the $L^2$-masses of the eigenfunctions come from negligible portions of the interval $(-L/2,L/2)$ while at energies of order $L$ or higher the eigenfunctions are spread over the whole interval.\\

In the present work, we complete the picture and provide detailed information on the spectrum of the infinite volume operator. To state the main result, let us first introduce the so-called Lyapunov exponent $\gamma_\lambda$, which is the (deterministic) exponential rate at which the solutions of $-y'' + \xi y = \lambda y$ grow at infinity, see Subsection \ref{Subsec:Lyapunov} for a precise definition. In the present situation, it happens that the Lyapunov exponent admits an explicit integral expression:
	$$ \gamma_\lambda := \frac{\int_0^\infty {\sqrt v} \exp(-2\lambda v - \frac{v^3}{6}) dv}{2\int_0^\infty \frac1{\sqrt v} \exp(-2\lambda v - \frac{v^3}{6}) dv}\;,\quad \lambda \in \R\;.$$
Our main result is as follows.
\begin{theorem}\label{Th:Expo}
	Almost surely, the spectrum of $\cH$ is pure point and equals $\R$, and for every eigenvalue $\lambda$ of $\cH$ the associated eigenfunction $\varphi_\gl$ satisfies 
	$$ - \frac{\ln (|\varphi_\gl(t)|^2 + |\varphi_\gl'(t)|^2)^{1/2}}{\vert t\vert} \to \gamma_\gl\;,\quad |t|\to\infty\;.$$
\end{theorem}

\bigskip

Our result establishes what is usually referred to as an \emph{Anderson localization} phenomenon: the spectral measure is pure point  and the associated eigenfunctions decay exponentially fast. Let us give a very brief overview of the literature on this topic. This phenomenon was originally predicted by the physicist Anderson~\cite{Anderson58} in 1958: he modelled the Hamiltonian of a quantum particle evolving in a crystal by the discrete Schr\"odinger operator $-\Delta + V$ on $\Z^d$ where $V=(V_k, k\in\Z^d)$ is a collection of i.i.d.~r.v.~representing the disorder induced by impurities or defects in the crystal. Anderson argued that, provided the random disorder $V$ is strong enough, the bottom of the spectrum of the Hamiltonian is no longer absolutely continuous, as for $-\Delta$, but rather pure point with exponentially localized eigenfunctions. The phenomenon can also be spelled out in the continuous setting, in which case $\Z^d$ becomes $\R^d$, $\Delta$ is the continuous Laplacian and $V$ is typically taken to be a random function, which is stationary and ergodic.

The first rigorous proof of Anderson localization was given in the continuous setting in dimension $1$ by Goldscheid, Molchanov and Pastur~\cite{GMP}, and subsequently many results~\cite{CarmonaDuke,kunz1980,FS,KotaniSimon,AM} were established both in discrete and continuous settings and for arbitrary dimensions. Let us summarize the main results: in dimension $d=1$, Anderson localization holds in the whole spectrum; in dimension $d \geq 2$, for a large enough disorder or at a low enough energy, Anderson localization holds. We refer to~\cite{Carmona,Kirsch} for more details and the main conjectures in this field.\\

Let us now comment on the disorder that we consider in the present work. Gaussian white noise arises as the scaling limit of fields of i.i.d.~r.v.~with finite variance: consequently, this is the natural continuum counterpart of the potentials considered in the discrete setting.\\
However, white noise is quite irregular as it is only \emph{distribution} valued: for instance, in dimension $1$, it can be obtained as the derivative in the sense of distributions of a Brownian motion. To the best of our knowledge, distribution valued potentials have never been considered in the Anderson localization literature and our setting is not covered by the aforementioned general results. As we will explain below, even the definition of the operator is unclear due to the irregularity of white noise. On the other hand, white noise is very convenient mathematically as it is totally uncorrelated and this allows for powerful tools from stochastic analysis: this is the reason why very detailed information can be obtained on this model.

\subsection*{Construction of the operator}

Heuristically, $\xi$ is irregular ``everywhere'' and this makes the identification of the domain of the operator non-trivial: whatever smoothness is imposed on $f$, the term $f\xi$ is distribution valued (except in the trivial case $f\equiv 0$) and does not belong to $L^2$. Consequently, the domain of $\cH$ cannot contain smooth functions; instead it should be made of functions $f$ which are such that $-f''$ is itself distribution valued but ``compensates'' for $f\xi$ in such a way that the sum $-f'' + f\xi$ lies in $L^2$.\\

In finite volume, one can avoid identifying precisely the domain by considering the Dirichlet form associated to $\cH_L$. The Dirichlet form requires \emph{less} regularity, and in the present case, it was shown by Fukushima and Nakao~\cite{Fukushima} that it is closed on the \emph{deterministic space} $H^1_0$, bounded below, and that it gives rise to a self-adjoint operator with pure point spectrum (the resolvents are compact). In infinite volume however, due to the ``unboundedness'' of white noise, the Dirichlet form is no longer bounded from below and does not allow to construct the operator anymore.\\

In the present work, we provide a detailed construction of the operator in finite and infinite volume through the theory of \emph{generalized Sturm-Liouville operators} developed by Weidmann~\cite{Weidmann}. The latter is an extension of the classical theory of Sturm-Liouville, that applies to operators of the form $-f'' + fV$ with $V$ in $L^1_{\mbox{\tiny{loc}}}$, to a setting that encompasses more singular potentials $V$. Actually, in this framework it implements the heuristic idea outlined above on the local behavior of elements of the domain, see Section \ref{Sec:FullSpace}. This leads to the following result:
\begin{theorem}\label{Th:Construction}
	Almost surely the operator $\cH$ is self-adjoint on (a dense subset of) $L^2(\R)$ and is the limit as $L\to\infty$, in the strong resolvent sense, of $\cH_L$.
\end{theorem}

Let us mention that in higher dimension, the operator $-\Delta + \xi$ is \emph{singular}: white noise is too irregular for the product $f\xi$ to make sense as soon as $\Delta f$ behave like $\xi$ at small scales. In dimensions $2$ and $3$, based on recent theories of stochastic PDEs~\cite{Max,Hairer2014}, a renormalization procedure allows to construct the operator and it formally writes $-\Delta + \xi + \infty$, see~\cite{AllezChouk, GUZ, Lab19, CvZ, Mouzard, BDM}.

\subsection*{Some comments on the proof of Anderson localization}

We present two proofs of Theorem \ref{Th:Expo}. The first proof relies on exponential decay estimates on the eigenfunctions of the finite volume operator $\cH_L$ that we establish uniformly over $L$ and that we combine with the vague convergence of the finite volume spectral measure to the infinite volume one. This strategy of proof is inspired by an article of Carmona~\cite{CarmonaDuke} on continuous Schr\"odinger operators in dimension $1$ (with a regular potential): let us mention that we prove that the Lyapunov exponent is not only an upper bound (as in Carmona's result) but actually the precise rate of exponential decay of the eigenfunctions. This first proof of Theorem \ref{Th:Expo} is presented in full detail.\\
We also present the main lines of an alternative proof which is based on a celebrated article of Kotani and Simon~\cite{KotaniSimon}. This approach is somehow more robust, but relies on ergodic theorems for products of $2\times 2$ random matrices and a priori growth estimates on generalized eigenfunctions. These intermediate results were not required in the first approach, and it turns out that the a priori growth estimates, that we establish in the present article, are far from being trivial. The ``usual'' proof of these estimates does not apply in our case due to the singularity of white noise. Instead, our proof exploits the relationship between the Hamiltonian $\cH$ and the evolution equation referred to as the parabolic Anderson model (PAM). 

\subsection*{Connection with the parabolic Anderson model} We establish rigorously the connection between $\cH$ and the PAM: it is more delicate than one may think, not only because white noise is singular, but also because the operator $\cH$ is \emph{not} bounded from below. Besides the growth estimates on the generalized eigenfunctions, we also derive that the spectrum of $\cH$ equals $\R$.

\subsection*{Stability} Throughout the article, we will discuss the stability of the results and proofs under a perturbation of the noise. We will focus on the case where the perturbation is actually a regularisation $\xi_\varepsilon$ of the white noise $\xi$ obtained by convolving $\xi$ with a smooth approximation of the Dirac delta that lives at scale $\varepsilon$. This is a standard procedure in stochastic PDE. This perturbed noise gives rise to the operator
$$ \cH^{(\varepsilon)} := -\partial_x^2 + \xi_\varepsilon\;,\quad \mbox{ on }\R\;.$$
We will see that the construction of the operator with the theory of generalized Sturm-Liouville operators is quite robust, and applies almost verbatim to $\xi_\varepsilon$. On the other hand, the first proof of Anderson localization that we present cannot easily be adapted to cover this case: the reason being that this proof relies on decay estimates that are not easy to establish in the non-Markovian setting induced by $\xi_\varepsilon$. The second proof of Anderson localization is more robust and applies to $\xi_\varepsilon$: however, it requires the positivity of the Lyapunov exponent, which is no longer explicit in this case, but this positivity is granted by a non-trivial result due to Kotani.

\subsection*{A general comment on the article} A complete proof of Anderson localization requires the combination of many different ingredients. There exist a few monographs on this topic, mostly in the discrete setting though, but they do not cover singular potential such as white noise.\\
We decided to take advantage of this article to give a detailed, and rather self-contained, presentation of many different aspects of the Schr\"odinger operator with white noise potential and a complete proof of Anderson localization. Some intermediate arguments are of course standard, but others are new: we avoided as much as possible to refer the reader to the literature and instead we tried to present all details.\\
We hope this article will be accessible to readers that are not experts of the spectral theory of self-adjoint operators nor of Anderson localization.

\subsection*{Organization of the article} In Section 2, we collect a few preliminary facts on the eigenproblem formally associated to the operator: the discussion concerns solutions to ODE and SDE only. In Section \ref{Sec:FullSpace}, we apply the theory of generalized Sturm-Liouville operators to define $\cH_L$ and $\cH$, and to prove Theorem \ref{Th:Construction}. In Section \ref{Sec:Expo}, we construct a tractable spectral measure for the operator $\cH$ and we prove the localization result of Theorem \ref{Th:Expo}. Section \ref{Sec:PAM} is devoted to the connection between the operator $\cH$ and the parabolic Anderson model: it can be read essentially independently from Section \ref{Sec:Expo}. The connection with the PAM will allow us to deduce that the spectrum of $\cH$ equals $\R$, thus completing the proof of Theorem \ref{Th:Expo} and will also allow us to prove a priori growth estimates on the generalized eigenfunctions, useful for the next section. Finally in Section \ref{Sec:Kotani}, we present the main steps of a second proof of Theorem \ref{Th:Expo} following the approach of Kotani and Simon. Each section from \ref{Sec:FullSpace} to \ref{Sec:Kotani} ends on a comment regarding the case where $\xi$ is replaced by $\xi_\varepsilon$.

\section{ODE, diffusions and Lyapunov exponent}

In this section, we collect some preliminary results on the solutions of the ODE
\begin{equation}\label{Eq:ODEz}
	-y'' + y\xi = z y\;,
\end{equation}
parametrized by $z \in \C$. To emphasize the dependence on $z$, we will sometimes write $y_z$. Although the rigorous construction of the operators $\cH_L$ and $\cH$ will be carried out only in the next section, we can already anticipate that this ODE, in the specific case where $z=\lambda \in\R$, is nothing but the eigenproblem associated to these operators.

Since $\xi$ is distribution-valued, $y''$ cannot be a function and \eqref{Eq:ODEz} has to be understood in the distribution sense. Nevertheless, for any given $z\in\C$, the set of solutions is a two-dimensional vector space: up to fixing two degrees of freedom, for instance $y(0)$ and $y'(0)$, there is a unique absolutely continuous solution to \eqref{Eq:ODEz}. We refer to Lemma \ref{Lemma:ODEz} for a precise (and more general) statement.

\subsection{Pr\"ufer coordinates and diffusions}\label{Subsec:Prufer}

Let us concentrate on the case where $z=\lambda \in \R$. It turns out that the evolution of the complex function $y_\lambda' + i y_\lambda$ has some non-trivial properties. It is convenient to work in polar coordinates, also called Pr\"ufer variables:
$$ y_\lambda' + i y_\lambda = r_\gl e^{i \theta_\gl} \;.$$
The process $\theta_\gl$ is called the \emph{phase} of $y_\gl$, and $r_\gl$ its \emph{modulus}. It is sometimes useful to deal with
$$ \rho_\gl(t) := \ln r_\gl^2(t)\;.$$
In these new coordinates, and writing $dB = \xi$, we have the following coupled stochastic differential equations:
\begin{align}
	d\theta_\gl(t) &= \big(1 + (\gl-1) \sin^2 \theta_\gl + \sin^3\theta_\gl \cos\theta_\gl\big) dt - \sin^2 \theta_\gl dB(t)\;, \label{eq_theta}\\
	d\rho_\gl(t) &= \big(-(\gl-1) \sin 2\theta_\gl - \frac12 \sin^2 2\theta_\gl + \sin^2 \theta_\gl \big)dt+ \sin 2\theta_\gl dB(t)\;.\label{EDSlnr}
\end{align}
The two degrees of freedom $y_\gl(0)$ and $y_\gl'(0)$ become $\theta_\gl(0)$ and $r_\gl(0)$ (or equivalently $\rho_\gl(0)$). Note that w.r.t.~the solution $y_\gl$ of \eqref{Eq:ODEz}, $r_\gl(0)$ is only a dilation coefficient, while $\theta_\gl(0)$ has a non-trivial impact on the solution. In the sequel, we will always take $\theta_\gl(0) = \theta_0 \in [0,\pi)$ and, most often, $r_\gl(0) = 1$ (or equivalently $\rho_\gl(0)=0$).
We also denote $\rho_{\gl,\theta_0}$ the solution of \eqref{EDSlnr} with initial conditions $\rho_\gl(0) =0$ and $\theta_\gl(0) = \theta_0$ and $r_{\gl,\theta_0}$ the corresponding modulus.\\

It can be checked that the stochastic process $\rho_\gl$ is completely determined by the knowledge of the trajectory of $\theta_\gl$. In addition, the phase satisfies the so-called Sturm-Liouville oscillation property: the process $t\mapsto \lfloor \theta_\gl(t)/\pi \rfloor$ is non-decreasing. In other words, the process $\theta_\gl(t)$ cannot cross downwards $0$, neither any $k\pi$ with $k\in \N$. Let us finally mention that the process $\theta_\gl$ mod$[\pi]$ remains Markovian and admits a unique invariant measure $\mu_\gl$, with an explicit density that we do not recall here.

\subsection{Lyapunov exponent and propagator}\label{Subsec:Lyapunov}

The growth of the modulus $r_\gl$ plays an important role in the spectral properties of the operator. This topic has a long history, that takes its roots in deep results on products of $2\times 2$ random matrices initiated by Furstenberg and Kesten. Due to its independence properties, the white noise potential ideally fits in this framework. In the first proof of Anderson localization that we will present in Section \ref{Sec:Expo}, we will not need the material below and we will rely instead on alternative estimates that were obtained in~\cite{DL_Crossover}. On the other hand, in the second proof of Anderson localization based on the Kotani-Simon approach, the material below will be crucial. In any case, we chose to present these facts in this preliminary section as we believe that they are quite helpful for understanding the exponential decay of the eigenfunctions.\\

For $\lambda \in \R$, consider the \emph{Neumann} and \emph{Dirichlet} solutions $y^{N}_\gl$ and $y^{D}_\gl$ of \eqref{Eq:ODEz}, that is, the solutions that satisfy
$$ y^N_\lambda(0) = 1\;,\quad (y^N_\lambda)'(0) = 0\;,\quad  y^D_\lambda(0) = 0\;,\quad (y^D_\lambda)'(0) = 1\;.$$
These two solutions span the two-dimensional vector space of solutions of \eqref{Eq:ODEz}. Introduce the propagator
 $$ U_\gl(t) := \begin{pmatrix} y_\gl^N(t) & y_\gl^D(t)\\ (y_\gl^N)'(t) & (y_\gl^D)'(t) \end{pmatrix}\;,\quad t\ge 0\;,$$
 At any time $t\ge 0$, this is an invertible matrix of determinant $1$ (this is because the Wronskian of $y^N_\lambda$ and $y^D_\lambda$ is constant, see Subsection \ref{Subsec:MinMax}). The process satisfies
 $$ U_\gl(t) = U_\gl(s,t) U_\gl(s)\;,\quad 0\le s \le t\;,$$
 where $U_\gl(s,t)$ is the propagator defined with the Neumann and Dirichlet solutions starting at time $s$. Note that $U_\gl(s,t)$ is independent of $U_\gl(s)$ and has the same law as $U_\gl(t-s)$. Restricting for convenience to integer times, this implies that $U_\gl(n)$ is the product of $n$ i.i.d.~$2\times 2$ random matrices. Furstenberg Theorem~\cite[Th 4.1 on p.30]{BougerolLacroix} ensures that almost surely
 $$ \frac1{n} \ln \| U_\gl(n)\| \to \gamma_\gl\;,\quad n\to \infty\;,$$
 where $\gamma_\gl$ is explicitly given by the so-called Furstenberg formula:
 $$ \gamma_\gl = \int_{\theta \in [0,\pi)}  \E[\ln r_{\gl,\theta}(1)] \mu_\gl(\theta) d\theta\;.$$

 \begin{lemma}
 	For any $\gl\in\R$, we have
 	$$ \gamma_\gl = \frac{\int_0^\infty {\sqrt v} \exp(-2\lambda v - \frac{v^3}{6}) dv}{2\int_0^\infty \frac1{\sqrt v} \exp(-2\lambda v - \frac{v^3}{6}) dv}\;.$$
 \end{lemma}
\begin{proof}
	From \eqref{EDSlnr}, we find
	$$ \E[\ln r_{\gl,\theta}(1)] = \frac12 \E\Big[\int_0^1 \big(-(\gl-1) \sin 2\theta_\gl(t) - \frac12 \sin^2 2\theta_\gl(t) + \sin^2 \theta_\gl(t) \big)dt\Big]\;.$$
	Integrating the previous expression against $\mu_\gl(\theta)d\theta$, which is the invariant law of the process $\theta_\gl$, we easily obtain
	$$ \gamma_\gl = \frac12 \int_{\theta\in [0,\pi)} \big(-(\gl-1) \sin 2\theta - \frac12 \sin^2 2\theta + \sin^2 \theta \big) \mu_\gl(\theta) d\theta\;.$$
	Then the computations in~\cite[Appendix A.4.2]{DL_Crossover} provide the desired expression.
\end{proof}

The extension of the above convergence from integer times to positive real times can be carried out in the following way. First note that we have not specified the norm on $2\times 2$ matrices that we work with: since they are all equivalent, we can take the maximum of the euclidean norms of its two column-vectors, namely $\| U_\gl(t) \| = \max(r^N_\gl(t), r^D_\gl(t))$. Second, for any $t\in [n,n+1]$
$$ \ln r^N_\gl(t) =  \ln r^N_\gl(n) +  \ln \frac{r^N_\gl(t)}{r^N_\gl(n)}\;.$$
Observe that $\frac{r^N_\gl(t)}{r^N_\gl(n)}$ is the euclidean norm at time $t$ of the solution of \eqref{Eq:ODEz} that starts at time $n$ with a phase equal to $\theta^N_\gl(n)$ and a modulus equal to $1$. If we set
$$X_n := \sup_{t\in [n,n+1]} \Big\vert \ln \frac{r^N_\gl(t)}{r^N_\gl(n)} \Big\vert\;,$$
then for any $n\ge 0$
$$ \E[X_n^2] \le \sup_{\theta_0\in [0,\pi)} \E\Big[\sup_{t\in [0,1]} \vert \ln r_{\gl,\theta_0}(t)\vert^2 \Big]\;.$$
From the SDE solved by $\ln r_{\gl,\theta_0}$, we can easily show that this last quantity is bounded. We thus deduce that for any $\varepsilon>0$
$$ \sum_{n\ge 1} \P(X_n / n > \varepsilon) < \infty\;,$$
which, by the Borel-Cantelli lemma, suffices to deduce that $X_n / n$ converges to $0$ almost surely. The same holds for the Dirichlet solution.\\
This almost sure convergence combined with the previous bounds allows to deduce that almost surely
$$ \lim_{t\to\infty} \frac1{t} \max(\ln r^N_\gl(t),\ln r^D_\gl(t)) = \lim_{n\to\infty} \frac1{n} \max(\ln r^N_\gl(n),\ln r^D_\gl(n))\;.$$
We have thus proved the following result:
\begin{lemma}[Furstenberg]
	For all $\gl \in \R$, almost surely
	\begin{equation}\label{Eq:Furstenberg}
		\frac1{t} \ln \| U_\gl(t)\| \to \gamma_\gl\;,\quad t\to \infty\;.
	\end{equation}
\end{lemma}

So far, we have shown that $\sup_{\theta_0}r_{\gl,\theta_0}(t)$ grows exponentially fast, but we do not know whether each $r_\gl(t,\theta_0)$ has such a behavior. Oseledec Theorem~\cite[Prop 1.1 on p.188-189]{BougerolLacroix} (which is a deterministic result) asserts that, on the event of probability one where \eqref{Eq:Furstenberg} holds, there exists an angle $\alpha_\gl \in [0,\pi)$ such that
$$ \frac{\ln r_{\gl,\alpha_\gl}(t)}{t} \to -\gamma_\gl\;,\quad t\to  +\infty\;,$$
and such that for all $\theta_0\in  [0,\pi)\backslash \{\alpha_\gl\}$
$$ \frac{\ln r_{\gl,\theta_0}(t)}{t} \to \gamma_\gl\;,\quad t\to + \infty\;.$$
In other words, there is one initial condition $\alpha_\gl$ for which the solution of \eqref{Eq:ODEz} decays exponentially fast at $+\infty$, while for any other initial condition the solution grows exponentially fast.\\
Each $\alpha_\gl$ gives rise to a square integrable solution, and therefore to a genuine eigenfunction of $\cH$. One could therefore think that this material is sufficient to characterize the nature of the spectrum of $\cH$. Note however that the existence of $\alpha_\gl$ is true up to an event of zero probability that \emph{depends} on $\lambda$, and it happens that this is \emph{not} true for all $\lambda$ simultaneously with probability one (see~\cite{CarmonaDuke}, end of Section 2).

\section{Construction of the operator}\label{Sec:FullSpace}

Our goal in this section is to construct a self-adjoint operator associated to the differential expression
$$\tau f := - f''  + f\xi \;.$$
In order to present the construction in a unified framework, we will work on a generic interval $(a,b) \subset \R$, with in mind either $(-L/2,L/2)$ or $\R$ itself. Of course, the nature of the operator depends drastically on whether $(a,b)$ is bounded or not, as we will see below. From now on, we let $B(t) := \langle \xi, \mathbf{1}_{[0,t]} \rangle$ be the Brownian motion associated to $\xi$.\\

The starting point of the construction consists in identifying a small set of functions on which the operator acts. Typically, one takes the set of infinitely differentiable functions with compact support in $(a,b)$. A difficulty arises in the present setting: since $\xi$ is irregular ``everywhere'', $-f''+ f\xi $ never lies in $L^2$ for any non-zero twice differentiable function $f$. As a consequence the small set of functions alluded to above is not obvious to identify.\\

The way out is the following elementary observation: an integration by parts prompts
$$\tau f = -f'' + f\xi = -(f'-Bf)' - f' B\;,$$
which suggests to look for functions $f$ which are not twice differentiable, but are such that $f$ and $f'-Bf$ are differentiable. This is in line with the informal discussion in the introduction that suggested that $f''$ should behave like $f\xi$ in order for $f$ to lie in the domain.\\

It turns out that the theory of generalized Sturm-Liouville operators allows to deal with such a \emph{distorted derivative}: in this section, we present the construction of the operator following the monograph by Weidmann~\cite{Weidmann} on this topic.\\

Before we start, let us introduce the Wronskian at $x$ of two functions $y_1$ and $y_2$:
\begin{align}\label{Wronskian}
	W_x(y_1,y_2) := y_1(x) y_2'(x) -  y_1'(x) y_2(x)\;.
\end{align}

Let us also state a general existence and uniqueness result for the ODE \eqref{Eq:ODEz}, and generalizations thereof. Let us mention that the only requirement on $\xi$ for this result to hold is that $t\mapsto B(t)$ lies in $L^2_{\mbox{\tiny loc}}$.

\begin{lemma}\label{Lemma:ODEz}
	Fix $z\in \C$ and $c\in\R$. For any $u,v \in \C$, there exists a unique function $y$ such that $y$ and $y'-By$ are absolutely continuous, satisfy $y(c) = u$, $y'(c) = v$ and
	\begin{equation}\label{Eq:ODEy}
		-(y'-By)' - y'B = z y\;.
	\end{equation}
	More generally, given any function $g\in L^1_{\mbox{\tiny loc}}$, the set of functions $y$ such that $y$ and $y'-By$ are absolutely continuous, and satisfy
	\begin{equation}\label{Eq:ODEyg}
		-(y'-By)' - y'B = z y + g\;,
	\end{equation}
	is a two-dimensional vector space. Given two solutions $y_1$ and $y_2$ of \eqref{Eq:ODEy} with a non-zero\footnote{The Wronskian of two solutions of \eqref{Eq:ODEy} is independent of the point it is evaluated at.} Wronskian $w := W(y_1,y_2)$, any solution $y$ of \eqref{Eq:ODEyg} can be written
	$$ w \; y(x) = y_2(x) \Big( \gamma + \int_c^x y_1(t) g(t) dt\Big) + y_1(x) \Big( \beta - \int_c^x y_2(t) g(t) dt\Big)\;,$$
	for some parameters $\beta,\gamma \in \C$.
\end{lemma}
\begin{proof}
	We only sketch the proof of the general statement involving $g$. Let us rewrite the equation in the matrix form
	\begin{align*}
		\begin{pmatrix}
			y\\
			y'-By
		\end{pmatrix}' = \begin{pmatrix}
		B&1\\
		-B^2-z&-B
	\end{pmatrix}\begin{pmatrix}
	y\\
	y'-By
\end{pmatrix} + \begin{pmatrix}
0\\
-g
\end{pmatrix}\;.
	\end{align*}
so that standard ODE arguments allow to conclude. The representation of the solution with the help of $y_1$ and $y_2$ follows from a simple computation.
\end{proof}

\subsection{The minimal and maximal operators}\label{Subsec:MinMax}

Let us define $\cD_{a,b}^{\max}$ as the set of complex-valued functions $f\in L^2(a,b)$ satisfying:\begin{enumerate}[label=(\roman*)]
	\item $f$ and $f'-Bf$ are absolutely continuous on $(a,b)$,
	\item $\tau f$ belongs to $L^2(a,b)$,
\end{enumerate}
and $\cD_{a,b}^c$ the set of all functions in $\cD_{a,b}^{\max}$ which are compactly supported in $(a,b)$. Since we are dealing with distorted derivatives, the density of these sets in $L^2(a,b)$ is not immediate, contrary to the case of a smooth potential.

\begin{lemma}\label{Lemma:Density}
	Fix $(a,b) \subset \R$. The set $\cD_{a,b}^c$, and therefore also the set $\cD_{a,b}^{\max}$, is dense in $L^2(a,b)$.
\end{lemma}

The only property on $\xi$ that is required for this result to hold is that $t\mapsto B(t)$ lies in $L^1_{\mbox{\tiny{loc}}}(a,b)$ (here it is even continuous). Indeed, this assumption ensures the following fact, which will be used in the proof: given $g\in L^1_{\mbox{\tiny{loc}}}(a,b)$, $v \in \C$, $c \in (a,b)$ and $z \in \C$, there exists a unique absolutely continuous function $f$ such that $f(c) = v$ and $f' - B f = g$.

\begin{proof}
	It suffices to show that $\cD_{a,b}^c$ is dense (for the $L^2$ norm) in the set of $\cC^1$ functions with compact support in $(a,b)$. Let $f$ be such a function and set $\varphi := f' - Bf$, which is also compactly supported. Let $c\in (a,b)$ be a point where $f$ vanishes and note that
	$$ f(t) = \int_c^t \varphi(s) e^{\int_s^t B(u) du} ds\;.$$
	To conclude that $f$ lies in $\cD_{a,b}^c$, we would need $f' - Bf$ to be absolutely continuous but this is not the case. We will therefore proceed by approximation.\\
	Let $(\varphi_n)_{n\ge 1}$ be a sequence of smooth functions with compact support that converges pointwise and in $L^2$ towards $\varphi$. For any $n$, let $g_n$ be the solution of $g_n' - Bg_n = \varphi_n$ that equals $0$ at $c$, namely
	$$ g_n(t) = \int_c^t \varphi_n(s) e^{\int_s^t B(u) du} ds\;.$$
	At this point, we observe that $g_n' - Bg_n$ is absolutely continuous. However $g_n$ is not necessarily compactly supported so it may not belong to $\cD_{a,b}^c$. Let $\chi$ be a compactly supported and smooth function which equals $1$ on the support of $f$. Set $f_n := \chi g_n$. Now $f_n$ is differentiable, $f'_n - B f_n = \chi \varphi_n + \chi' f_n$ is also differentiable and $\tau f_n$ lies in $L^2(a,b)$. From its explicit expression, it is clear that $f_n$ converges to $f$ pointwise and in $L^2$, thus concluding the proof.
\end{proof}

\medskip

An important identity for the sequel is \emph{Green's formula} which states that for all $f,g \in \cD_{a,b}^{\max}$ and for any $a<c<d<b$:
$$\int_c^d \overline{g(x)} \;\tau f(x) dx  = W_c(\bar{g}, f) - W_d(\bar{g},f) + \int_c^d \overline{ \tau g(x)} \; f(x) dx \;.$$
The integrability properties satisfied by $f,g$ allow to pass to the limit $c\downarrow a$ and $d\uparrow b$ and to deduce that $W_a(\bar{g}, f)$ and $W_b(\bar{g}, f)$ are well-defined. We thus obtain
\begin{equation}\label{Eq:WaWb}
	\int_a^b  \overline{g(x)} \;\tau f(x) dx = W_a(\bar{g}, f) - W_b(\bar{g},f) + \int_a^b \overline{ \tau g(x)} \; f(x) dx\;.
\end{equation}

We now let $\cH^{\max}_{a,b}$ be the operator $\tau$ acting on the domain $\cD_{a,b}^{\max}$. The next result will imply that it is a closed operator, usually referred to as the \emph{maximal operator}. Let also $\cH^{c}_{a,b}$ be the operator $\tau$ acting on the domain $\cD_{a,b}^{c}$. This operator is not closed, and we thus let $(\cH^{\min}_{a,b},\cD^{\min}_{a,b})$ be its closure: it is referred to as the \emph{minimal operator}. The following proposition gives a characterization of $\cD^{\min}_{a,b}$ and identifies the adjoints of the previously defined operators. For convenience, we omit the proof and refer to~\cite[Lemma 9.4]{Teschl}.

\begin{proposition}\label{Prop:MinMax}
	We have 
	\begin{align*}
		(\cH_{a,b}^c)^\ast =(\cH_{a,b}^{\min})^\ast = \cH^{\max}_{a,b}\,.
	\end{align*}
	Moreover, the domain of the minimal operator writes:
	\begin{align*}
		\cD^{\min}_{a,b} &= \{u \in \cD^{\max}_{a,b}\;:\: \quad W_a(v,u) - W_b(v,u) = 0 \quad \mbox{ for all }\; v \in \cD^{\max}_{a,b}\}\,,\\
		&=  \{u \in \cD^{\max}_{a,b}\;:\: \quad W_a(v,u) = W_b(v,u) = 0 \quad \mbox{ for all }\; v \in \cD^{\max}_{a,b}\}\,.
	\end{align*}
\end{proposition}
\begin{remark}
	As previously mentioned, this proposition implies that the operator $(\cH^{\max}_{a,b},\cD^{\max}_{a,b})$ is closed and that the assumption (i) imposed on elements in $\cD^{\max}_{a,b}$ is ``correct'': namely, one does not need to relax this assumption and include more functions.
\end{remark}

\subsection{Weyl's alternative}\label{Subsec:Weyl}

Our goal is now to investigate the possible self-adjoint extensions of $\cH^{\min}_{a,b}$. By the result of Proposition \ref{Prop:MinMax}, it necessarily lies in between $\cH^{\min}_{a,b}$ and $\cH^{\max}_{a,b}$.\\

A prominent role will be played by the solutions of the ODE \eqref{Eq:ODEz} parametrized by $z\in \C$. As stated in Lemma \ref{Lemma:ODEz}, the set of solutions is a two-dimensional complex vector space, and of course any two solutions $y_1$ and $y_2$ form a basis provided they are linearly independent. It is easy to check that for any two solutions $y_1$ and $y_2$, the value of their Wronskian \eqref{Wronskian} does not depend on $x$ and vanishes if and only if they are colinear.\\

Here comes an important dichotomy. We say that $\tau$ is
\begin{itemize}
	\item Limit circle at $a$ if for all $z\in \C$, all solutions $y$ of \eqref{Eq:ODEz} are square-integrable at $a$,
	\item Limit point at $a$ if for all $z\in \C$, there exists a solution $y$ of \eqref{Eq:ODEz} that is not square-integrable at $a$.
\end{itemize}
The same definition is taken at $b$.\\
One may think that these two cases do not cover all possible situations as the number of square-integrable solutions may depend on $z$. Actually, this is a famous result, often referred to as Weyl's alternative, that these two situations are complementary. More precisely, Weyl's alternative states that: if for some $z\in \C$, there are two linearly independent solutions $y_1$ and $y_2$ that are square-integrable at $a$ then for all $z$, all solutions are square-integrable at $a$. The proof is elementary and can be found in~\cite[Theorem 5.6]{Weidmann}, but also in~\cite[Lemma III.1.2]{Carmona} and~\cite[Theorem 9.9]{Teschl}.\\

In the present situation, it is immediate to check that, almost surely, we are in the limit circle case at $a$ (resp.~$b$) whenever $a$ (resp.~$b$) is finite: indeed all solutions of \eqref{Eq:ODEz} are locally bounded since $B$ is locally in $L^2$ (thanks to Lemma \ref{Lemma:ODEz}). On the other hand, the behavior at infinity requires a little argument:
\begin{lemma}
	Fix $\lambda\in \R$. Almost surely, the solution $y$ of \eqref{Eq:ODEz} satisfying $y(0)=1$ and $y'(0)=0$ is integrable neither at $+\infty$ nor at $-\infty$. As a consequence, almost surely we are in the limit point case at both infinities.
\end{lemma}
\begin{proof}
	By symmetry, it suffices to prove the property at $+\infty$. We use the Pr\"ufer coordinates introduced in Subsection \ref{Subsec:Prufer}. Consider the diffusion $(\theta_\gl,r_\gl)$ starting from $\theta_\gl(0) = \pi/2$ and $r_\gl(0) = 1$. Introduce for every $k\ge 0$, $T_k := \inf\{t\ge 0: \theta_\gl(t)=(2k+1)\pi/2\}$ and $S_k := \inf\{t\ge T_k: \vert \theta_\gl(t) - (2k+1)\pi/2 \vert = \pi/4\}$. Then, the r.v.~$(S_k-T_k)_{k\ge 0}$ are i.i.d.~positive r.v. By~\cite[Lemma 7.2]{DL_Crossover}, (note that $\gamma_\gl = \nu_\gl/2$) for any given $\varepsilon > 0$, there exists a random $C > 0$ such that for all $t\ge 0$
	$$ r_\gl(t) \ge C e^{(\gamma_\gl - \varepsilon)t}\;.$$
	Since for $t\in [T_k,S_k]$, $\sin^2 \theta_\gl(t) \ge 1/2$ we have
	\begin{align*}
		\int_0^{+\infty} y_\gl(t)^2 dt =  \int_0^{+\infty} r_\gl(t)^2 \sin^2 \theta_\gl(t) dt \ge   \frac{C^2}{2} \sum_{k\ge 0} (S_k-T_k) e^{2(\gamma_\gl - \varepsilon)T_k}\;.
	\end{align*}
	Almost surely this series diverges since $(S_k-T_k) e^{2(\gamma_\gl - \varepsilon)T_k}$ does not converge to $0$ as $k\to\infty$.
\end{proof}

Let us now explain the denominations ``limit circle'' and ``limit point''. To simplify the presentation, let us restrict to the present situation where $\tau$ is limit point only at $\pm \infty$. Below we let $e_\alpha$ denote the unit vector in $\R^2$ whose coordinates are $\sin \alpha$ and $\cos \alpha$. For any vector $v\in \C^2$, we write $v\parallel e_\alpha$ if $v$ and $e_\alpha$ are colinear, that is, if $v = c e_\alpha$ for some $c\in \C$. With a slight abuse of notation, we will write $y(c) \parallel e_\alpha$ or ``$y$ is colinear to $e_\alpha$ at $c$'' as a shortcut for $(y(c),y'(c)) \parallel e_\alpha$.\\

Fix $b>0$ and some $z\in \C\backslash \R$. Recall that thanks to Lemma \ref{Lemma:ODEz}, there exist unique solutions $y^{N}$ and $y^{D}$ of \eqref{Eq:ODEz} that satisfy
$$ y^N(0) = 1\;,\quad (y^N)'(0) = 0\;,\quad  y^D(0) = 0\;,\quad (y^D)'(0) = 1\;.$$
The letters $N$ and $D$ stand for Neumann and Dirichlet and any solution of \eqref{Eq:ODEz} is a linear combination of $y^N$ and $y^D$. Note that $y^N$ (resp.~$y^D$) cannot be colinear to some $e_\beta$ at $b$: if it were, then $W_b(\bar{y}^{N}, y^N)$ would vanish and since by definition of $y^N$, $W_0(\bar{y}^{N}, y^N)$ already vanishes, identity \eqref{Eq:WaWb} would imply that
$$ {z} \int_0^b  \vert y^N \vert^2 = \int_0^b \overline{y^{N}} \tau y^{N} = \int_0^b \overline{\tau y^{N}}  y^N = \bar{z} \int_0^b  \vert y^N \vert^2\;,$$
a contradiction since $z$ is non-real.\\
As a consequence, for any $\beta \in [0,\pi)$, there exists a unique complex number $m_z(b,\beta) \neq 0$ such that, if we set
$$ y := y^N + m_z(b,\beta) y^D\;,$$
then $y$ is colinear to $e_\beta$ at $b$.\\
Then it can be checked that as $\beta$ varies over $[0,\pi)$, $m_z(b,\beta)$ describes a circle $\ccC(b)$ in the complex plane whose radius equals
$$ \frac1{2 \vert \Im(z) \vert  \int_0^b \vert y^D \vert^2}\;.$$
It can also be checked that as $b$ increases the circles $\ccC(b)$ decrease in the sense that for $b < b'$, the circle $\ccC(b')$ lies inside the circle $\ccC(b)$. Therefore, as $b\to\infty$ $\ccC(b)$ either converges to a circle or to a point.

If it converges to a point, the radius converges to $0$ and therefore $y^D$ is not square integrable at $+\infty$: $\tau$ is then limit point at $+\infty$. The arguments actually show that if we denote by $m_z(+\infty)$ the point that arises as the limit of the circles $\ccC(b)$ as $b\to\infty$, then $y^N + m_z(+\infty) y^D$ is the unique (up to multiplicative factor) solution of \eqref{Eq:ODEz} that is square integrable at $+\infty$.

If it converges to a circle, then the radius remains positive so $y^D$ is square integrable at $+\infty$, and it can be checked that all solutions of \eqref{Eq:ODEz} are integrable at infinity: $\tau$ is then limit circle at $+\infty$. (In this last case, the argument shows moreover that for any $\beta \in [0,\pi)$ there exists a (necessarily) integrable solution of \eqref{Eq:ODEz} that is colinear to $e_\beta$ but we will not use this fact).\\

The rest of the presentation now considers separately the case where the interval $(a,b)$ is bounded (limit circle case) and the case where $(a,b) = \R$ (limit point case). In the first case, we will construct self-adjoint extensions of the minimal operator by imposing some boundary conditions at $a$ and $b$. In the second case, we will see that there exists a unique self-adjoint extension.

\subsection{Limit circle case}

We start with the limit circle case at both boundaries, that holds whenever $(a,b)$ is a bounded interval. In the particular case where $(a,b)=(-L/2,L/2)$ and $\alpha = 0$, we get the operator $\cH_L$ of the introduction. Recall that $f(a) \parallel e_\alpha$ is a shortcut for $(f(a),f'(a)) \parallel e_\alpha$.

\begin{proposition}\label{Prop:Circle}
	Fix $\alpha \in [0,\pi)$ and some bounded interval $(a,b) \subset \R$. The operator
	$$ \cH_{a,b}^\alpha f := \tau f\;,$$
	on the domain
	\begin{align*}
		\cD_{a,b}^\alpha := \Big\{& f\in \cD_{a,b}^{\max}: f(a) \parallel e_\alpha \mbox{ and } f(b) \parallel e_\alpha \Big\}\;,
	\end{align*}
	is a self-adjoint operator. For any $z\in \C\backslash \R$, let $y^a$ and $y^b$ be two solutions of \eqref{Eq:ODEz} satisfying
	$$ y^a(a) \parallel e_\alpha \;,\quad y^b(b) \parallel e_\alpha\;.$$
	These two solutions are linearly independent and they yield the following kernel for the resolvent of $\cH_{a,b}^\alpha$. For any $g\in L^2(a,b)$
	$$ (\cH_{a,b}^\alpha - z)^{-1} g(s) = \int_a^b G_{a,b}^\alpha(z,s,t) g(t) dt\;,$$
	with\footnote{Recall that the Wronskian of two solutions of \eqref{Eq:ODEz} does not vary with the point it is evaluated at.}
	$$ G_{a,b}^\alpha(z,s,t) := \frac1{W(y^a,y^b)} \begin{cases} y^b(s) y^a(t) &\mbox{ if } s\ge t\;,\\
		y^b(t) y^a(s) &\mbox{ if } s \le t\;.
	\end{cases}
	$$
	The resolvents are Hilbert-Schmidt operators, and the spectrum of $\cH_{a,b}^\alpha$ is discrete, bounded below and accumulates at $+\infty$.
\end{proposition}
The proof follows the lines of~\cite[Th 9.6 and Lemma 9.7]{Teschl}: for completeness, we provide it.
\begin{proof}
	It is easy to check that $\cH_{a,b}^\alpha$ is a symmetric extension of $\cH_{a,b}^{\min}$, meaning that $\cD^{\min}_{a,b} \subset \cD_{a,b}^\alpha$ and $\cH_{a,b}^\alpha$ is symmetric on $\cD_{a,b}^\alpha$. Therefore its adjoint satisfies $(\cH_{a,b}^\alpha)^* \subset \cH_{a,b}^{\max}$. By \eqref{Eq:WaWb}, we have for any $g$ in the domain of $(\cH_{a,b}^\alpha)^*$
	$$ W_a(\bar{g}, f) - W_b(\bar{g},f)  = 0\;,\quad \forall f \in \cD_{a,b}^\alpha\;.$$
	Given any function $f \in \cD_{a,b}^\alpha$, one can multiply $f$ by a smooth cutoff function that equals $1$ near $a$ and vanishes near $b$: the product hence obtained is itself in $\cD_{a,b}^\alpha$ and consequently
	$$ W_a(\bar{g}, f) = W_b(\bar{g},f)  = 0\;,\quad \forall f \in \cD_{a,b}^\alpha\;.$$
	Since any $f\in \cD_{a,b}^\alpha$ satisfies the b.c.~imposed at $a$ and $b$, we deduce that $\bar{g}$ must also satisfy these b.c. This ensures that $g$ belongs to $\cD_{a,b}^\alpha$. We thus deduce that $\cH_{a,b}^\alpha$ is self-adjoint.\\
	We pass to the expression of the kernel of the resolvent. The existence of $y^a$ and $y^b$ is trivial: it suffices to consider the solutions of \eqref{Eq:ODEz} satisfying the right b.c.~at either $a$ or $b$. If they were linearly dependent, there would exist some complex number $c\ne 0$ such that $y^b= c y^a$ and we would get $ W_a(\bar{y^b},y^a) = W_b(\bar{y^b},y^a) = 0$ so that by \eqref{Eq:WaWb} we would have
	$$ z \bar{c} \int_a^b |y^a|^2 = \bar{z} \bar{c} \int_a^b |y^a|^2\;,$$
	and $z$ would thus necessarily be real, a contradiction. We thus set $w := W(y^a,y^b) \ne 0$.\\
	Let us now prove the identity of the statement for any function $g \in L^2(a,b)$ with compact support: by a density argument and the boundedness of the operators at stake, this is sufficient for our purpose. Set $f := (\cH_{a,b}^\alpha - z)^{-1} g$ so that $(\tau-z)f = g$. By Lemma \ref{Lemma:ODEz}, $f$ can be expressed as a linear combination of $y^a$ and $y^b$, namely for some complex numbers $\beta,\gamma$
	$$ wf(x) = y^b(x) \Big( \gamma + \int_a^x y^a(t) g(t) dt\Big) + y^a(x) \Big( \beta + \int_x^b y^b(t) g(t) dt\Big)\;.$$
	Since $g$ vanishes near $a$, we have $wf(x) =  y^b(x) \gamma +  y^a(x) \Big( \beta + \int_a^b y^b(t) g(t) dt\Big)$ near $a$ so that
	$$ W_a(y^a,wf) = \gamma W_a(y^a,y^b) = \gamma w\;.$$
	Recall that $f = (\cH_{a,b}^\alpha - z)^{-1} g$ so that it belongs to $\cD_{a,b}^\alpha$. Hence $f$ is colinear to $y^a$ at $a$ and this forces $\gamma = 0$. The same argument at $b$ shows that $\beta = 0$.\\
	Since the kernel lies in $L^2((a,b)^2,ds\otimes dt)$, the resolvent is Hilbert-Schmidt and therefore compact. Let us now show that the operator is bounded below. We concentrate on the case where $\alpha = 0$ for simplicity. An integration by parts shows that for any $f\in\cD^\alpha_{a,b}$
	$$ \langle f,-(f'-Bf)'\rangle = \langle f', f'-Bf\rangle\;,$$
	so that
	$$ \langle f, \cH^\alpha_{a,b} f\rangle = \langle f, -(f'-Bf)' - Bf'\rangle = \langle f',f'\rangle - 2\langle f',Bf \rangle $$
	Using the bound
	$$ 2\vert \langle f',Bf\rangle\vert \le \|f'\|_{L^2(a,b)}^2 + \sup_{t\in (a,b)} \vert B(t) \vert^2 \|f \|_{L^2(a,b)}^2\;,$$
	we deduce that for $M>0$ large enough
	$$ \langle f, \cH^\alpha_{a,b} f\rangle + M \|f\|_{L^2(a,b)}^2 > 0\;,$$
	which ensures that the spectrum is bounded below. Since $0$ is not an eigenvalue of the compact operator $(\cH_{a,b}^\alpha - z)^{-1}$, we deduce that the operator $\cH^\alpha_{a,b}$ has discrete spectrum which accumulates at $+\infty$.
\end{proof}

\subsection{Limit point case}

We pass to the limit point case at both boundaries, which in our setting holds only when $(a,b) = \R$. For simplicity, we write $\cD_{\R}^{\cdots}$ instead of $\cD_{-\infty,+\infty}^{\cdots}$. Let us start with a technical result.

\begin{lemma}
	For all $g,h \in \cD^{\max}_{\R}$, we have $W_{-\infty}(g,h) = 0$ and $W_{+\infty}(g,h) = 0$.
\end{lemma}
\begin{remark}
	This result is not at all specific to the operator at stake.
\end{remark}
\begin{proof}
	By contradiction, suppose there exist $g,h \in \cD^{\max}_{\R}$ such that $W_{-\infty}(g,h) \ne 0$. Since the Wronskian is bilinear, without loss of generality we can assume that $g$ and $h$ are real. Take some arbitrary point $c\in \R$ and consider the following two domains
	\begin{align*}
		\cD^1 &:= \{ f\in \cD^{\max}_{-\infty,c}: W_{-\infty}(f,g) = W_c(f,g) = 0\}\;,\\
		\cD^2 &:= \{ f\in \cD^{\max}_{-\infty,c}: W_{-\infty}(f,h) = W_c(f,g) = 0\}\;.
	\end{align*}
	In other words, we impose colinearity with $g$ at $-\infty$ in the first case, and colinearity with $h$ at $-\infty$ in the second case. Furthermore in both cases, we impose colinearity with $g$ at $c$ but this is arbitrary.\\
	These two sets are distinct since, for instance, $g$ belongs to $\cD^1$ but not to $\cD^2$. Now, following the very same arguments as in the proof of the last proposition, we see that the operators $\cH^1$ and $\cH^2$ naturally defined on these two domains are self-adjoint. We will get a contradiction by showing that their resolvents coincide.\\
	Fix $z\in \C\backslash\R$. Let $y^-$ be the unique (up to multiplicative factor) solution of \eqref{Eq:ODEz} that is square integrable near $-\infty$ (uniqueness is due to $\tau$ being limit point at $-\infty$). Let also $y^+$ be a solution of \eqref{Eq:ODEz} which is colinear to $g$ at $c$. The Wronskian $w$ of $y^-$ and $y^+$ is necessarily non-zero (otherwise they would be colinear, and the same argument as in the previous proof would imply that $z$ is real). Take $g_0\in L^2(\R)$ with compact support and set for $i\in\{1,2\}$, $f^i := (\cH^i - z)^{-1} g_0 \in \cD^i$. Necessarily $(\tau-z)f^i = g_0$ so that Lemma \ref{Lemma:ODEz} ensure that there exist $\gamma^i, \beta^i \in \C$ such that
	$$ w \, f^i(x) = y^+(x) \big( \gamma^i + \int_{-\infty}^x y^-(t) g_0(t) dt\big) + y^-(x) \big(\beta^i + \int_x^c y^+(t) g_0(t) dt \big)\;.$$
	For $x$ close to $-\infty$, since $g_0$ vanishes we must have
	$$ wf^i(x) = \gamma^i  y^+(x) + y^-(x) \big(\beta^i  + \int_{-\infty}^c y^+(t) g_0(t) dt \big)\;.$$
	As we are in the limit point case at $-\infty$, $y^+$ cannot be square integrable at $-\infty$. The square integrability of $f^i$ implies that $\gamma^i = 0$ and therefore, there exists $c^i$ such that
	$$ f^i(x) = c^i y^-(x)\;.$$
	This implies that $f^1$ and $f^2$ are both colinear to $y^-$ at $-\infty$, so $g$ and $h$ must be colinear to each other at $-\infty$, a contradiction.
\end{proof}

We thus deduce that $\cD^{\min}_{\R} = \cD^{\max}_{\R}$ and therefore the minimal and maximal operators coincide.

\begin{proposition}\label{Prop:H}
	The operator $\cH f = -f'' + f\xi$ on the domain $\cD = \cD^{\min}_{\R} = \cD^{\max}_{\R}$ is self-adjoint. The set $\cD^c_\R$ is a core for $\cH$, that is, the closure of $(\cH^c_\R,\cD^c_\R)$ equals $(\cH,\cD)$. For any $z\in \C\backslash \R$, let $y^\pm$ be the two solutions of \eqref{Eq:ODEz} that are square-integrable at $\pm \infty$. Necessarily these two solutions are linearly independent and they yield the following kernel for the resolvent of $\cH$: for any $g\in L^2(\R)$
	$$ (\cH - z)^{-1} g(s) = \int_\R G(z,s,t) g(t) dt\;,$$
	with
	$$ G(z,s,t) := \frac1{W(y^{-},y^{+})} \begin{cases} y^{+}(s) y^{-}(t) &\mbox{ if } s\ge t\;,\\
		y^{+}(t) y^{-}(s) &\mbox{ if } s \le t\;.
	\end{cases}
	$$
\end{proposition}
\begin{proof}
	By the previous lemma, Proposition \ref{Prop:MinMax} and identity \eqref{Eq:WaWb}, we see that $\cH$ on $\cD$ is self-adjoint. By construction, $\cD^c_\R$ is a core for $(\cH,\cD)$. The identification of the kernel of the resolvent follows from exactly the same arguments as in the limit circle case.
\end{proof}

We can proceed to the proof of Theorem \ref{Th:Construction}. Let $\cH_L$ be the operator $\cH_{a,b}^\alpha$ with interval $(a,b) = (-L/2,L/2)$ and $\alpha = 0$. This operator and its resolvent $(\cH_L-z)^{-1}$ are operators defined on the Hilbert space $L^2(-L/2,L/2)$ whereas the limiting operator acts on $L^2(\R)$. This issue can be easily circumvented by viewing $\cH_L$ and  $(\cH_L-z)^{-1}$ as operators on $L^2(\R)$ by projecting $L^2(\R)$ on $L^2(-L/2,L/2)$ in the canonical way and by extending any element of $L^2(-L/2,L/2)$ into an element of $L^2(\R)$ by setting its values to $0$ outside $(-L/2,L/2)$. By a slight abuse of notation, we will still denote by $\cH_L$ and  $(\cH_L-z)^{-1}$ those extensions to $L^2(\R)$.

\begin{proof}[Proof of Theorem \ref{Th:Construction}]
 It remains to show that $\cH_L$ converges in the strong resolvent sense to $\cH$, namely for any $z\in\C\backslash\R$, we aim at showing that for all $g\in L^2(\R)$, $(\cH_L-z)^{-1}g$ converges in $L^2$ to $(\cH-z)^{-1}g$ as $L\to\infty$.\\
We could prove the convergence using the explicit expressions of the kernels of the resolvents together with the properties of the Weyl functions collected before. We prefer to present a shorter and more elegant argument due to Weidmann~\cite{Weidmann}. Let $f \in \cD^c_\R$. For $L$ large enough, the support of $f$ is fully included in $(-L/2,L/2)$ and we deduce that $f\in \cD_{-L/2,L/2}^c \subset \cD_{-L/2,L/2}$, the latter being the domain of $\cH_L$. In particular, $\cH_L f = \cH f$.\\
For any $z\in \C\backslash\R$ set $g := (\cH-z) f$ and note that
$$ (\cH_L-z)^{-1} g - (\cH-z)^{-1} g =  (\cH_L-z)^{-1} (\cH -z + z- \cH_L) (\cH-z)^{-1} g = (\cH_L-z)^{-1}(\cH-\cH_L) f = 0\;.$$
Recall that $\cH-z$ is a bijection between $\cD$ and $L^2(\R)$. By definition of a core, $\cD^c_\R \times (\cH-z) \cD^c_\R$ is dense in $\cD \times L^2(\R)$. Since in addition the operators at stake are bounded (by $1/|\Im(z)|$) we deduce the asserted strong resolvent convergence.
\end{proof}

\begin{remark}[Perturbation of the noise]
	The arguments presented in this section are very robust. Only two properties of the noise were used. First, we relied on the local integrability of $B(t) = \langle \xi, \un_{[0,t]}\rangle$: this remains obviously true if we replace $\xi$ by $\xi_\varepsilon$. Second, our operator is limit point at both infinities and we gave a rather ad-hoc argument to ensure this property. With $\xi_\varepsilon$, one can check that the operator remains limit point at both infinities by applying a general criterion for potentials that are function-valued~\cite[Cor. III.1.5]{Carmona}. Namely one needs to check that there exists a random constant $C > 0$ such that
	$$ \xi_\varepsilon(t) > - C (1+t^2)\;,\quad \forall t\in \R\;,$$
	and this property is not difficult to establish.
\end{remark}

\section{Anderson localization}\label{Sec:Expo}

Our main theorem deals with the spectrum of $\cH$ and the associated spectral measure. The definition of the former is standard: this is the complement of the set of values $E\in \C$ for which $\cH-E \;:\; \cD \to L^2(\R)$ admits a (necessarily bounded) inverse. Since $\cH$ is self-adjoint, the spectrum is a subset of $\R$.\\
On the other hand, the notion of spectral measure is probably less standard: in the following paragraph, we aim at clarifying this point for the non-expert reader. This requires to introduce the notion of projection-valued measure.\\

A projection-valued measure (p.v.m.) is a map $E$ from the Borel set $\mathcal{B}(\R)$ into the set of orthogonal projections on $L^2(\R)$ that satisfies (i) $E(\emptyset) = 0$ and $E(\R) = Id$ ; (ii) for any collection of disjoints Borel sets $I_n, n\ge 1$, it holds $E(\cup_n I_n) = \sum_{n} E(I_n)$ where the series converges strongly.\\

Given a p.v.m. $E$, one can introduce spectral measures associated to elements in $L^2(\R)$ in the following way. For $f\in L^2(\R)$, the measure
$$ \mu_f(d\lambda) := \langle f, E(d\lambda) f\rangle\;,$$
is usually referred to as the spectral measure associated to $f$. It is non-negative and its total mass equals the $L^2$-norm of $f$. Similarly one can define the complex spectral measure associated to $f,g \in L^2(\R)$:
$$ \mu_{g,f}(d\lambda) := \langle g, E(d\lambda) f\rangle\;.$$

Given a locally bounded, measurable function $h:\R\to\R$ one obtains a self-adjoint operator by setting\footnote{Actually the right-hand-side needs to be defined. When $h = \mathbf{1}_I$ is the indicator of a Borel set $I$, then this is simply the self-adjoint operator $E(I)$. This can be extended by density arguments, see for instance~\cite[Chap 3.1]{Teschl}}
\begin{equation}\label{Eq:Eh}
	E(h) := \int_\R h(\lambda) \,E(d\lambda)\;,
\end{equation}
defined on the domain $\{f \in L^2(\R),\; \int |h(\gl)|^2  \mu_f(d\lambda) < \infty\}$. Note that when $h: \R \to \C$, the operator can be defined again on this domain, but is not necessarily self-adjoint. Moreover when $h$ is bounded, $E(h)$ is defined on the whole $L^2(\R)$ and is bounded. \\

At several occasions, we will apply the following property (see e.g.~\cite[Th 3.1]{Teschl} ) linking the operator $E(h)$ to the spectral measures $\mu_{g,f}$: for any bounded measurable map $h:\R\to\C$ and any $f,g \in L^2$ we have
\begin{equation}\label{Eq:ELuELv}
	\langle g, E(h) f \rangle_{L^2} = \int_\gl h(\gl) \mu_{g,f}(d\lambda)  \;,
\end{equation}

The spectral theorem asserts that \emph{all} self-adjoint operators are of the form $E(h)$ for some p.v.m. $E$ and some locally bounded function $h$, and that \emph{uniqueness} of the p.v.m.~holds provided $h = id$. Therefore almost surely there exists a unique p.v.m.~$E$ satisfying
$$ \cH = \int_{\R} \lambda\, E(d\lambda)\;.$$

To prove that $E$ corresponds to the p.v.m of $\cH$, one can check that for all $z \in \C\backslash \R$, the following equality holds:
\begin{align*}
	E\Big(\frac{1}{\cdot -z}\Big) = (\cH - z)^{-1}\,,
\end{align*}
as it is easy to see that  the p.v.m. $E$ is characterized by the family $1/(\cdot -z)$, $z \in \C \backslash \R$.

\medskip

A \emph{spectral measure} is a non-negative Radon measure $\sigma$ on $\R$ that is equivalent to $E(d\lambda)$ in the following sense: for any Borel set $A$, $E(A) = 0$ if and only if $\sigma(A) = 0$. Of course, all spectral measures are equivalent to each other and we can talk of the pure point (resp.~absolutely continuous, singular) component of the spectrum of $\cH$ to denote the pure point (resp.~absolutely continuous, singular) component of any spectral measure.

\medskip

For a given $f \in L^2(\R)$, the spectral measure associated to $f$ is not a spectral measure in general as its support may be a strict subset of the spectrum of $\cH$. The general theory ensures the existence of some $f\in L^2$ such that $\mu_f$ is a spectral measure (its support coincides with the spectrum of $\cH$), and it thus suffices to study this single measure to determine the nature of the spectrum. However, in practice, this $f$ is not accessible. In the present setting of second order differential operator, one builds a spectral measure $\sigma$ by setting (formally)
$$ \sigma := \mu_{\delta_0} + \mu_{\delta_0'}\;,$$
where $\delta_0'$ is the derivative in the sense of distributions of $\delta_0$. This writing is completely formal as neither $\delta_0$ nor $\delta_0'$ lie in $L^2(\R)$, and the rigorous construction of this measure is a non-trivial, although classical, task that we will carry out by passing to the limit on the finite volume counterpart of the measure $\sigma$. \\

In finite volume, the p.v.m.~associated to $\cH_L$ simply writes:
\begin{equation}\label{Eq:EL}
	E_L(d\lambda) := \sum_{k\ge 1} \langle \varphi_{k,L},\cdot \rangle \varphi_{k,L}\, \delta_{\gl_{k,L}}(d\lambda)\;,
\end{equation}
where $(\lambda_{k,L},\varphi_{k,L})_k$ are the eigenvalues/eigenfunctions of $\cH_L$. We can define the finite volume counterpart of $\sigma$ by setting:
\begin{equation}\label{Eq:sigmaL}
	\sigma_L := \sum_{k\ge 1} (\varphi_{k,L}(0)^2 + \varphi_{k,L}'(0)^2) \delta_{\lambda_{k,L}}(d\lambda)\;.
\end{equation}
It is easy to check that $\sigma_L$ is indeed a spectral measure for $\cH_L$. The measure $\sigma$ is obtained by passing to the limit on $L$ as the following result shows.

%

\begin{proposition}\label{Prop:Sigma}
	Almost surely, as $L\to\infty$ the measure $\sigma_L$ converges vaguely to some random Radon measure that we denote by $\sigma$ and which is a spectral measure for $\cH$.
\end{proposition}

Therefore, we will prove that $\cH$ is pure point by showing that $\sigma$ is a purely atomic measure. The main technical step in that direction is the following. Below we consider the processes $(r_{\gl,\theta_0},\theta_\gl)$ introduced in Subsection \ref{Subsec:Prufer} starting from $\theta_\gl(0) = \theta_0 \in [0,\pi)$ and $r_{\gl,\theta_0}(0) = 1$. Recall that $\theta_\gl$ satisfies an autonomous equation, while the equation satisfied by $r_{\gl,\theta_0}$ depends on the evolution of $\theta_\gl$. We need the following decay estimate.

\begin{proposition}\label{Prop:Expo}
	For any bounded interval $\Delta \subset \R$ and for any $\epsilon > 0$, there exists $q = q(\Delta,\epsilon) > 0$ such that
	$$ \E\bigg[ \int_\Delta \inf_{\theta_0\in [0,\pi)} \sup_{t\in \R} \Big(r_{\gl,\theta_0}(t) e^{(\gamma_\gl-\epsilon)|t|}  + \frac1{r_{\gl,\theta_0}(t)} e^{-(\gamma_\gl+\epsilon)|t|} \Big)^q d\sigma(\gl) \bigg] < \infty\;.$$
\end{proposition}

With the last two propositions at hand, and anticipating the result of Proposition \ref{Prop:Support}, the proof of Theorem \ref{Th:Expo} is relatively simple.

\begin{proof}[Proof of Theorem \ref{Th:Expo}]
	Up to taking a countable collection of bounded intervals $\Delta_n$ that covers $\R$, the bound of Proposition \ref{Prop:Expo} shows that, almost surely, $\sigma$-almost all $\gl \in \R$ is an eigenvalue of $\cH$, and since the collection of eigenvalues is at most countable, the measure $\sigma$ is pure point. By Proposition \ref{Prop:Sigma}, we deduce that a.s.~$\cH$ is pure point. 
	
	Since $\cH$ is limit point the following property holds. Almost surely, for every $\gl \in \R$ there is at most one $\theta_0\in [0,\pi)$ such that $\| r_{\gl,\theta_0}(\cdot)\|_{L^2(\R)} < \infty$, and if such a $\theta_0$ exists then the corresponding $y_\gl$ belongs to $\cD(\cH)$, and therefore $\gl$ is an eigenvalue of $\cH$ and $y_\gl$ is an eigenfunction associated to $\gl$. Combined with the previous paragraph, it implies that almost surely each eigenvalue has multiplicity one.
	By Proposition \ref{Prop:Expo}, almost surely, for every eigenvalue $\lambda \in \R$ there is a random constant $C > 0$ such that
	$$  \sup_{t\in \R} \bigg( \frac{\big(\varphi_\gl(t)^2 + \varphi_\gl'(t)^2\big)^{1/2}}{\big(\varphi_\gl(0)^2 + \varphi_\gl'(0)^2\big)^{1/2}} e^{(\gamma_\gl-\epsilon)|t|}  + \frac{\big(\varphi_\gl(0)^2 + \varphi_\gl'(0)^2\big)^{1/2}}{\big(\varphi_\gl(t)^2 + \varphi_\gl'(t)^2\big)^{1/2}} e^{-(\gamma_\gl+\epsilon)|t|} \bigg) < C\;.$$
	This implies the existence of $C'>0$ such that for all $t\in \R$
	$$ \frac1{C'} e^{-(\gamma_\gl+\epsilon)|t|} \le \Big( \varphi_\gl(t)^2 + \varphi_\gl'(t)^2 \Big)^{1/2} \le C' e^{-(\gamma_\gl-\epsilon)|t|}\;,$$
	and, since $\epsilon > 0$ is arbitrary the desired asymptotic behavior of the eigenfunctions is proved. The fact that the spectrum equals $\R$ will be proven in Proposition \ref{Prop:Support}. The proof of the theorem is therefore complete.
\end{proof}

The next three subsections are devoted to the proof of Proposition \ref{Prop:Sigma}. We proceed in two steps. First, we present the spectral analysis (matrix spectral measure, p.v.m., unitary transformation) of the operator in finite volume. Second, we show that the infinite volume limit of the matrix spectral measure allows to build the p.v.m.~of $\cH$, and we then prove the proposition. The presentation could be made shorter but we believe that a detailed presentation of the finite volume case greatly clarifies the discussion. The fourth subsection is devoted to the proof of Proposition \ref{Prop:Expo}.

\subsection{Spectral analysis in finite volume}\label{Subsec:FV}

For $z \in \C$, denote by $y^N_z(\cdot)$ and $y^D_z(\cdot)$ the Neumann and Dirichlet solutions of the ODE \eqref{Eq:ODEz} on $\R$, that is, the solutions that satisfy
$$ y^N_z(0) = 1\;,\quad (y^N_z)'(0) = 0\;,\qquad \mbox{and}\qquad y^D_z(0) = 0\;,\quad (y^D_z)'(0) = 1\;.$$
For every eigenvalue $\gl_{k,L}$ of $\cH_L$, the associated ($L^2$-normalized) eigenfunction can be written
\begin{equation}\label{Eq:phikDN}
\varphi_{k,L}(\cdot) = \varphi_{k,L}(0) y^N_{\gl_{k,L}}(\cdot) + \varphi_{k,L}'(0) y^D_{\gl_{k,L}}(\cdot)\;.
\end{equation}
Recall from Proposition \ref{Prop:Circle} that the spectrum of $\cH_L$ is discrete and that $\gl_{k,L} \to +\infty$ as $k\to\infty$. We then set
$$ \xi_L := \sum_{k\ge 1} \varphi_{k,L}(0)^2 \delta_{\gl_{k,L}}\;,\quad \zeta_L := \sum_{k\ge 1} \varphi_{k,L}'(0)^2 \delta_{\gl_{k,L}}\;,\quad \eta_L := \sum_{k\ge 1} \varphi_{k,L}(0)\varphi_{k,L}'(0) \delta_{\gl_{k,L}}\;.$$
These Radon measures allow to define the matrix spectral measure
$$M_L := \begin{pmatrix} \xi_L & \eta_L\\
\eta_L & \zeta_L
\end{pmatrix}\;.$$
It is easy to check that for any Borel set $A\subset \R$, we have $\xi_L(A),\zeta_L(A)\ge 0$ and $2|\eta_L(A)| \le \xi_L(A) + \zeta_L(A)$. As a consequence, each measure appearing in $M_L$ is absolutely continuous w.r.t.~the trace of $M_L$, which is nothing but the measure $\sigma_L$ introduced in \eqref{Eq:sigmaL}.

We then consider the set $L^2(\R,dM_L)$ of all measurable functions $F:\R\to\R^2$ that satisfy
$$ \langle F,F \rangle_{L^2(dM_L)} := \int_\R F(\gl)^\intercal \; M_L(d\gl) F(\gl) < \infty\;.$$

\begin{remark}
There is a little ambiguity in this definition as the meaning of the integral on the r.h.s.~must be specified. Define the matrix $N_L$ of the Radon-Nikodym derivatives of $M_L$ w.r.t.~$\sigma_L$. This is a non-negative symmetric matrix for $\sigma_L$-almost all $\gl$. Then
$$ \int_\R F(\gl)^\intercal \; M_L(d\gl) F(\gl) := \int_\R F(\gl)^\intercal \; N_L(\gl) F(\gl) \sigma_L(d\gl) < \infty\;.$$
The integral on the r.h.s.~is interpreted in the usual sense: this is the integral of an $\R$-valued function against the measure $\sigma_L$ on $\R$. It can be checked that $\langle \cdot,\cdot \rangle_{L^2(dM_L)}$ defines an inner product that turns $L^2(\R,dM_L)$ into a Hilbert space, see~\cite[Prop. I.4.5]{Carmona} for instance.
\end{remark}

We now introduce a unitary map from $L^2(-L/2,L/2)$ into $L^2(\R,dM_L)$, which should be thought of as a Fourier transform. For any function $f$, we set
\begin{equation}\label{Eq:FNL}
	F^N_L(\gl) := \int_{-L/2}^{L/2} f(x) y^N_\gl(x) dx\;,\quad F^D_L(\gl) := \int_{-L/2}^{L/2} f(x) y^D_\gl(x) dx\;.
\end{equation}
\begin{lemma}\label{Lemma:UL}
The map $U_L : L^2(-L/2,L/2) \to L^2(\R,dM_L)$ defined by
$$ U_L f = \begin{pmatrix} F^N_{L} \\ F^D_L\end{pmatrix}\;,$$
is unitary.
\end{lemma}
\begin{proof}
From the eigen-decomposition associated to the operator $\cH_L$ we find
$$ \|f\|_{L^2(-L/2,L/2)}^2 = \sum_{k} \langle \varphi_{k,L}, f\rangle^2\;,$$
and we use the decomposition \eqref{Eq:phikDN} of $\varphi_{k,L}$ on the Neumann and Dirichlet solutions to obtain
\begin{align*}
\langle \varphi_{k,L}, f\rangle^2 &= F^N_L(\gl_{k,L})^2 \xi_L(\{\gl_{k,L}\}) + F^D_L(\gl_{k,L})^2 \zeta_L(\{\gl_{k,L}\}) + 2  F^N_L(\gl_{k,L})F^D_L(\gl_{k,L}) \eta_L(\{\gl_{k,L}\})\\
&= U_Lf(\gl_{k,L})^\intercal \; M_L(\{\gl_{k,L}\}) U_Lf(\gl_{k,L})\;.
\end{align*}
Summing up over $k$, we obtain
$$ \|f\|_{L^2(-L/2,L/2)}^2 = \langle U_Lf , U_Lf \rangle_{L^2(dM_L)}\;.$$
\end{proof}

We then introduce the map $E_L$ from $\cB(\R)$ into the set of bounded operators on $L^2(-L/2,L/2)$ by associating to any $I \in \cB(\R)$ the orthogonal projection
$$ E_L(I) = U_L^{-1} \mathbf{1}_I U_L\;.$$

\begin{lemma}
	The map $E_L$ is the unique p.v.m.~associated to the operator $\cH_L$.
\end{lemma}
\begin{remark}
	Of course $E_L$ could have been defined through \eqref{Eq:EL}. However, in the infinite volume case this will not be possible anymore since, at this point of the proof, one does not have a spectral decomposition of the operator $\cH$. Consequently, we preferred to present a construction of the p.v.m.~which is robust when passing to infinite volume.
\end{remark}
\begin{proof}
	It is elementary to check that $E_L$ is indeed a p.v.m. To conclude, it suffices to check that for $z\in \C\backslash\R$
	$$ (\cH_L-z)^{-1} = E_L\Big(\frac1{\cdot - z}\Big)\;,$$
	where the notation $E_L(h)$ for some bounded measurable function $h:\R\to\C$ was introduced in \eqref{Eq:Eh}.\\
	Fix $f,g \in L^2(-L/2,L/2)$. First note the following equality between signed measures
	\begin{equation}\label{Eq:EML}
		\langle g, E_L(d\gl) f \rangle_{L^2(-L/2,L/2)} = (U_Lg(\gl))^\intercal M_L(d\gl) (U_Lf(\gl))\;.
	\end{equation}
	(It suffices to evaluate the two measures on any given Borel set).
	
	This being given, we compute
	\begin{align}
		\langle g,(\cH_L-z)^{-1}f \rangle_{L^2(-L/2,L/2)} &= \sum_{k} \frac1{\lambda_k-z} \langle g,\varphi_{k,L}\rangle \langle f,\varphi_{k,L}\rangle \notag \\
		&= \int_\R \frac1{\lambda-z} \begin{pmatrix} G^N_L & G^D_L
		\end{pmatrix} \begin{pmatrix} \xi_L(d\gl) & \eta_L(d\gl) \notag \\
			\eta_L(d\gl) & \zeta_L(d\gl)
		\end{pmatrix} \begin{pmatrix} F^N_L  \\
			F^D_L
		\end{pmatrix} \notag \\
		&= \int_\R \frac1{\lambda-z} (U_L g)^\intercal \; M_L(d\gl) (U_L f)\notag \\
		&= \int_\R \frac1{\lambda-z} \langle g, E_L(d\gl) f\rangle_{L^2(-L/2,L/2)}\;, \label{eq:resspectral}
	\end{align}
	where the last line follows from \eqref{Eq:EML}.\\

Using moreover the property \eqref{Eq:ELuELv} with the function $1/(\cdot -z)$, we get:
	\begin{align*}
		\langle g,(\cH_L-z)^{-1}f \rangle_{L^2(-L/2,L/2)} &= \langle g , E_L\Big(\frac1{\cdot - z}\Big) f\rangle_{L^2(-L/2,L/2)}\;.
	\end{align*}
	Since it holds for all $f,g \in L^2(-L/2,L/2)$, this allows to conclude.
\end{proof}

\subsection{The projection-valued measure in infinite volume}

To extend the previous construction to infinite volume, we need the analogue of $M_L$ for the operator $\cH$. Note that the definition of $M_L$ rests on the facts that the spectrum of $\cH_L$ is discrete and locally finite; these two properties being themselves a consequence of the compactness of the resolvents of $\cH_L$. At that point of the proof, we do not have any such information about $\cH$ (and actually, we will see that, although its spectrum is pure point, its eigenvalues are dense in $\R$ and its resolvents are not compact) so that there is no direct analogue of $M_L$ in infinite volume. However, the limit as $L\to\infty$ of $M_L$ is a natural candidate.

\begin{lemma}\label{Lemma:M}
Almost surely, the matrix spectral measure $M_L$ converges vaguely to some limit $M = \begin{pmatrix} \xi & \eta\\
\eta & \zeta
\end{pmatrix}$. This limit satisfies the following property: for any Borel set $A\subset \R$,
$$ 2 |\eta(A)| \le \xi(A) + \zeta(A)\;,\quad\text{ and } \quad \xi(A), \zeta(A) \ge 0\;.$$
Moreover $\sigma := \xi + \eta$ is a Radon measure.
\end{lemma}

This is a standard result of the literature. A complete proof can be found in~\cite[Chap.2 Sec.6]{Levitan}: it consists in showing tightness of the collection of measures $(M_L)_{L>0}$ by some ad-hoc argument (that needs to be adapted to cover our case where the potential is singular) and then identifying the limit. In~\cite[Chap.3 p.105]{Carmona}, an alternative proof, based on the convergence of the resolvents, is suggested (for the operator on a half-line with a non-singular potential): we have not found the details anywhere and it happens that there are a few subtleties along the way so we present a complete proof of Lemma \ref{Lemma:M} in Subsection \ref{Subsec:Stieltjes} following this alternative approach.\\

With the matrix $M$ at hand, we proceed as in the previous subsection to identify a spectral representation of $\cH$. We introduce the Hilbert space $L^2(\R,dM)$ of all measurable functions $F:\R \to \R^2$ satisfying
$$ \langle F,F\rangle_{L^2(dM)} := \int_{\R} F(\lambda)^\intercal \; dM(\lambda) F(\lambda) < \infty\;.$$

This being given, we introduce a unitary transform from $L^2(\R)$ to $L^2(\R,dM)$. However, the infinite volume setting induces the very same difficulties as for the Fourier transform: namely one cannot define $F^N$ and $F^D$ by \eqref{Eq:FNL} with $(-L/2,L/2)$ replaced by $\R$, and so one needs to proceed by approximation. Recall the maps 
\begin{align*}
	f\mapsto (F^N_n,F^D_n) := \Big(\int_{-n/2}^{n/2} f(x)y_\gl^N(x)dx,  \int_{-n/2}^{n/2} f(x)y_\gl^D(x)dx\Big)\;,
\end{align*}
introduced in \eqref{Eq:FNL}.

\begin{proposition}\label{Prop:Parseval}
For any $f\in L^2(\R)$, $(F^N_n,F^D_n)_{n\ge 1}$ is a Cauchy sequence in $L^2(\R,dM)$. Its limit is denoted $(F^N,F^D)$. The map $U:L^2(\R) \to L^2(\R,dM)$ defined by
$$ Uf = \begin{pmatrix} F^N \\ F^D\end{pmatrix}\;,$$
is unitary.
\end{proposition}
Note that $F^N, F^D$ are formally given by
$$ F^N(\gl) := \lim_{n\to\infty} \int_{-n}^{n} f(x) y^N_\gl(x) dx\;,\quad F^D(\gl) := \lim_{n\to\infty} \int_{-n}^{n} f(x) y^D_\gl(x) dx\;.$$
\begin{proof}
Let $f \in \cD^c_\R$. Such a function is compactly supported and belongs to $L^2(\R)$. Hence $F^N_n$ and $F^D_n$ are independent of $n$ provided $n$ is large enough, and therefore their limits $F^N$ and $F^D$ are well-defined, continuous functions of $\gl$. Our first goal is to show that $Uf$ belongs to $L^2(dM)$ and that its norm coincides with the norm of $f$.\\
Provided the support of $f$ belongs to $(-L/2,L/2)$, the eigen-decomposition associated to the operator $\cH_L$ yields
$$ \|f\|_{L^2(\R)}^2 = \sum_{k} \langle \varphi_{k,L}, f\rangle^2\;.$$
Fix $\mu > 0$ and split the above sum into the contributions coming from $\lambda_{k,L} \in [-\mu,\mu]$ and the rest. Regarding the latter, we observe that since $\varphi_{k,L}$ and $f$ belong to the domain of $\cH_L$, self-adjunction yields
$$\langle \varphi_{k,L}, f\rangle = \frac1{\lambda_{k,L}} \langle \cH_L \varphi_{k,L}, f\rangle = \frac1{\lambda_{k,L}} \langle \varphi_{k,L}, \cH_L f\rangle\;.$$
Since in addition $f$ belongs to the domain $\cD$ of $\cH$, we have $\cH_L f = \cH f \in L^2(\R)$. Consequently
\begin{align*}
\sum_{k:|\lambda_{k,L}|>\mu} \langle \varphi_{k,L}, f\rangle^2 &\le \frac1{\mu^2}\sum_{k:|\lambda_{k,L}|>\mu} \langle \varphi_{k,L}, \cH f\rangle^2\\
&\le \frac1{\mu^2} \sum_{k}\langle \varphi_{k,L}, \cH f\rangle^2 = \frac1{\mu^2} \|\cH f\|_{L^2(\R)}^2\;.
\end{align*}
This quantity goes to $0$ as $\mu\to\infty$. Let us now deal with the contribution coming from $\lambda_{k,L} \in [-\mu,\mu]$. By the same argument as in the proof of Lemma \ref{Lemma:UL} and since $U_L f = U f$ we get
\begin{align*}
\sum_{k:|\lambda_{k,L}| \le \mu} \langle \varphi_{k,L}, f\rangle^2 &= \sum_{k:|\lambda_k| \le \mu} U f(\gl_{k,L})^\intercal \;  M_L(\{\gl_{k,L}\}) U f(\gl_{k,L})\\
&= \int_{\gl: |\gl| \le \mu}  U f(\lambda)^\intercal \; M_L(d\lambda) U f(\lambda)\;.
\end{align*}
Whenever $\mu$ and $-\mu$ are not atoms of the limiting measures $\xi,\zeta,\eta$, the vague convergence of $M_L$ to $M$ ensures that this last quantity converges to
\begin{align*}
\int_{\gl: |\gl| \le \mu} U f(\lambda)^\intercal \; M(d\lambda) U f(\lambda)= \big\| \mathbf{1}_{|\gl| \le \mu} Uf  \big\|^2_{L^2(dM)} \;.
\end{align*}
Henceforth, we have proven that
$$ \Big\vert \|f\|_{L^2(\R)}^2 - \| \mathbf{1}_{|\gl| \le \mu} Uf \|^2_{L^2(dM)} \Big\vert \le \frac1{\mu^2} \|\cH f\|_{L^2(\R)}^2\;.$$
Passing to the limit on an appropriate sequence of $\mu \to \infty$, we thus deduce that $Uf \in L^2(dM)$ and that
$$  \|f\|_{L^2(\R)} = \| Uf \|_{L^2(dM)}\;.$$

It remains to extend this identity to the whole set $L^2(\R)$. This is an easy consequence of the density of $\cD^c_\R$ in $L^2(\R)$. Indeed for any $f\in L^2(\R)$ there exists a sequence $f_n$ of elements of $\cD^c_\R$ that converges to $f$. The linearity of $U$ ensures that
$$  \|f_n - f_m\|_{L^2(\R)} = \| Uf_n - Uf_m \|_{L^2(dM)}\;,$$
so that $(Uf_n)_n$ is a Cauchy sequence in $L^2(dM)$ and we can easily conclude that the norm of its limit coincides with the norm of $f$.
\end{proof}

We can now introduce the map $E$ that associates to any Borel set $I$ the orthogonal projection
$$ E(I) := U^{-1} \mathbf{1}_I\, U \;.$$

\begin{proposition}\label{propo:Epvm}
Almost surely, $E$ is the unique p.v.m.~associated to $\cH$.
\end{proposition}
\begin{proof}
	Is is elementary to check that $E$ is a p.v.m. To complete the proof, it suffices to show that $(\cH-z)^{-1} = E(\frac1{\cdot - z})$ for any given $z\in\C\backslash\R$, which in turn is equivalent to showing that for any $f,g \in L^2(\R)$ we have
	$$ \langle g, (\cH-z)^{-1}  f \rangle_{L^2(\R)} = \langle g, E(\frac1{\cdot - z}) f\rangle_{L^2(\R)}\;. $$
	By density and boundedness of the operators $(\cH-z)^{-1}$ and $E(\frac1{\cdot - z})$, it is sufficient to establish the identity for all $f,g \in \cD^c_\R$. Fix $f,g \in \cD^c_\R$ and assume from now on that $L$ is taken large enough for their supports to be included in $(-L/2,L/2)$. Consequently $U_L f = U f$ and $U_L g = U g$. Recall from \eqref{eq:resspectral}:
\begin{align*}
\langle g, (\cH_L-z)^{-1}  f \rangle_{L^2(-L/2,L/2)} 
&=\int_\R \frac1{\lambda-z} \big(U g(\lambda)\big)^\intercal \; M_L(d\lambda) \big(U f(\lambda)\big)\;.
\end{align*}
We split this integral into the contributions coming from $|\gl| > \mu$ and the rest. Cauchy-Schwarz inequality gives
\begin{align*}
\Big\vert \int_{|\gl| > \mu} \frac1{\lambda-z} U g(\lambda)^\intercal \; M_L(d\lambda) U f(\lambda) \Big\vert&=\Big\vert \Big\langle \mathbf{1}_{|\cdot| > \mu}\frac1{\cdot - z} Ug , \mathbf{1}_{|\cdot| > \mu} Uf \Big\rangle_{L^2(dM_L)} \Big\vert\\
&\le \frac1{\vert \Im(z)\vert} \Big\| \mathbf{1}_{|\cdot| > \mu} Ug \Big\|_{L^2(dM_L)}\, \big\| \mathbf{1}_{|\cdot| > \mu} Uf \big\|_{L^2(dM_L)}\;.
\end{align*}
We saw in the proof of Proposition \ref{Prop:Parseval} that
$$ \Big\| \mathbf{1}_{|\cdot| > \mu} Uf \Big\|_{L^2(dM_L)}^2= \sum_{k:|\lambda_{k,L}|>\mu} \langle \varphi_{k,L}, f\rangle^2 \le  \frac1{\mu^2} \|\cH f\|_{L^2(\R)}^2\;,$$
and this quantity converges to $0$ as $\mu\to \infty$, uniformly over $L$ large enough. The same holds for $Ug$.\\
On the other hand, provided $\mu$ and $-\mu$ are not atoms of the measures $\xi,\zeta,\eta$, the almost sure vague convergence of $M_L$ towards $M$ ensures that
$$ \int_{\gl \in [-\mu,\mu]} \frac1{\lambda-z}\big(U g(\lambda)\big)^\intercal \;  M_L(d\lambda) \big(U f(\lambda)\big)\;,$$
converges as $L\to\infty$ to
$$ \int_{\gl \in [-\mu,\mu]} \frac1{\lambda-z}U g(\lambda)^\intercal \;  M(d\lambda) U f(\lambda)\;.$$
As $\mu\to \infty$, this quantity converges towards
$$\int_{\gl \in \R} \frac1{\lambda-z} U g(\lambda)^\intercal \;  M(d\lambda) U f(\lambda)\;.$$
At this point, we observe that the following equality of measures holds
$$ \big(U g(\lambda)\big)^\intercal \;  M(d\lambda) \big(U f(\lambda)\big) = \langle g, E(d\lambda) f\rangle_{L^2(\R)}\;.$$
Consequently, using \eqref{Eq:ELuELv} at the second line, we obtain
\begin{align*}
\int_{\gl \in \R} \frac1{\lambda-z} U g(\lambda) M(d\lambda) U f(\lambda) &= \int_{\gl \in \R} \frac1{\lambda-z} \langle g, E(d\lambda) f\rangle_{L^2(\R)}\\
 &= \langle g, E(\frac1{\cdot - z}) f\rangle_{L^2(\R)}\;.
\end{align*}
We have therefore proven that $\langle g, (\cH_L-z)^{-1}  f \rangle_{L^2(-L/2,L/2)}$ converges to $\langle g, E(\frac1{\cdot - z}) f\rangle_{L^2(\R)}$ as $L\to\infty$. Combined with the strong convergence of $(\cH_L-z)^{-1}$ towards $(\cH-z)^{-1}$, we deduce the identity
$$ \langle g, (\cH-z)^{-1}  f \rangle_{L^2(\R)} = \langle g, E(\frac1{\cdot - z}) f\rangle_{L^2(\R)}\;,$$
thus concluding the proof.
\end{proof}

We can now proceed with the proof of Proposition \ref{Prop:Sigma}.

\begin{proof}[Proof of Proposition \ref{Prop:Sigma}] Recall that $\sigma := \mbox{Tr}(M) = \xi + \eta$.
	We have already shown that a.s.~$M_L$ converges vaguely to $M$, so in particular $\sigma_L$ converges vaguely to $\sigma$. It remains to show that $\sigma$ is a spectral measure. To that end, it suffices to show that for any Borel set $A$,
	$$ \Big(\mu_f(A) = 0\;,\quad \forall f\in L^2(\R) \Big)  \Leftrightarrow \sigma(A) = 0\;.$$
	For any $f\in L^2(\R)$, we have
	$$ \mu_f(d\lambda) = \langle f, E(d\lambda) f\rangle = \big(Uf(\gl)\big)^\intercal \;  M(d\lambda) \big(Uf(\gl)\big)\;.$$
	From Lemma \ref{Lemma:M}, all measures appearing in $M$ are absolutely continuous w.r.t.~$\sigma$, and we easily deduce that $\mu_f$ is itself absolutely continuous w.r.t.~$\sigma$. Conversely, assume that $\sigma(A) > 0$ and let us prove the existence of $f$ such that $\mu_f(A) > 0$. Without loss of generality, we can assume that $A$ is bounded. Let us introduce the matrix $N$ as the Radon-Nikodym derivative of $M$ w.r.t.~$\sigma$. The trace of $N$ is constant equal to $1$ $\sigma$-a.e. Consequently, $N(\lambda) \ne 0$ $\sigma$-a.e.~and we can therefore find $v(\lambda) \in \R^2$ such that $v(\lambda)^\intercal \;  N(\lambda) v(\lambda) = 1$ for $\sigma$-a.a.~$\lambda$. To conclude, it suffices to take $f := U^{-1} (\mathbf{1}_A v)$ ($\mathbf{1}_A v$ is in $L^2(dM)$ as $A$ is bounded).
\end{proof}

\subsection{Convergence of the matrix spectral measure}\label{Subsec:Stieltjes}

In the previous subsection, we left aside the proof of Lemma \ref{Lemma:M} on the convergence of the matrix spectral measure. As we already explained, we present a proof which is suggested in~\cite{Carmona} but whose details are more subtle that one may expect at first sight. In particular, it requires some arguments on Stieltjes transforms of Radon measures combined with the following version of Mercer's Theorem:

\begin{lemma}[Mercer's Theorem]\label{Theorem:Mercer}
	Let $I$ be an interval of $\R$. Assume that $K(s,t)$ is a continuous kernel from $I\times I$ to $\C$ that lies in $L^2(I \times I)$, and that there exists a collection $f_n$, $n\ge 1$ of continuous functions from $I$ to $\C$ such that for any $g,h$ in a dense subset of $L^2(I)$
	\begin{align}\label{dec_kernel_Mercer}
\langle g, K h\rangle = \sum_{n\ge 1} \int \bar{g}(s)\bar{f_n}(s) ds \int h(t) f_n(t) dt\;,
	\end{align}
	where the series converges absolutely.
	Then we have $K(s,t) = \sum_{n\ge 1} \bar{f_n}(s) f_n(t)$ where the sum on the right hand side converges absolutely pointwise, and locally uniformly.
\end{lemma}

\begin{proof}
First note that the form of the kernel implies that $K$ is a non-negative kernel that is $\langle g, K g\rangle \geq 0$ for all $g \in L^2(I)$. Moreover, for any $n\ge 1$, the kernel
	$$ K(s,t) - \sum_{k=1}^n \bar{f_k}(s) f_k(t)\;,$$
	is jointly continuous and non-negative. Consequently for all $n \geq 0$,
	$$ \sum_{k=1}^n \vert f_k \vert^2(s) \le K(s,s)\;.$$

	For any $n_0\ge 1$, Cauchy-Schwarz inequality yields
	\begin{align*}
		\sum_{n\ge n_0} \vert \bar{f_n}(s) f_n(t) \vert& \le \big(\sum_{n\ge n_0} \vert f_n \vert^2(s) \big)^{1/2} \big(\sum_{n\ge n_0} \vert f_n \vert^2(t) \big)^{1/2}\\
		&\leq \big(K(s,s))^{1/2} \big(\sum_{n\ge n_0} \vert f_n \vert^2(t) \big)^{1/2}\,.
	\end{align*}

It implies that for all $s, t \in I$, $ \sum_{n\ge 1} \bar{f_n}(s) f_n(t)$ converges absolutely pointwise and locally uniformly in $s$ for fixed $t$ (and locally uniformly in $t$ for fixed $s$). Denote by $\tilde K(s,t)$ this point-wise limit which is necessarily continuous in $s$ for all $t$ fixed and continuous in $t$ for $s$ fixed.

Moreover let $B \subset I$ be compact, thanks to pointwise convergence and domination by the continuous function $K(s,s) K(t,t)$, we have that $(\sum_{k=0}^{n} \bar{f_k}(s) f_k(t))_n$ converges in $L^2(B \times B)$ necessarily towards $\tilde{K}(s,t)$. From \eqref{dec_kernel_Mercer}, we deduce that $\langle g,\tilde{K} h\rangle = \langle g, Kh \rangle$ for all $g,h$ in a dense set of $L^2(B)$. From the boundedness of $K$ and $\tilde{K}$ on $B\times B$, this remains true for all $g,h \in L^2(B)$ so that $\tilde{K} = K$ Lebesgue-a.e.~on $B\times B$. The continuity of $K$ and the partial continuity of $\tilde{K}$ are sufficient to deduce that $K=\tilde{K}$ everywhere on $B\times B$.
	By Dini's Theorem, that we can apply as $K(s,s)$ is continuous, the convergence of
	$$ s\mapsto \sum_{n\ge 1} \vert f_n\vert^2(s)\;,$$
	is locally uniform over $I$. Cauchy-Schwarz again implies that the convergence of $ \sum_{n \geq 1} \bar{f_n}(s) f_n(t)$ is locally uniform as well.
\end{proof}

Let us now collect some definitions and results on the Stieltjes transform of Radon measures. Let $\mu$ be a Radon measure on $\R$ satisfying
\begin{equation}\label{Eq:Radon}
	\int_\R \frac1{1+\vert \lambda \vert } \mu(d\lambda) < \infty\;.
\end{equation}
Set $\C_+:= \{z:\Im(z) >0\}$. The Stieltjes transform of $\mu$ is the map $f:\C_+ \to \C_+$ defined through
$$ f(z) := \int_\R \frac1{\lambda - z} \mu(d\lambda)\;.$$
\begin{lemma}\label{Lemma:Stieltjes}
	Let $\mu_n$, $n\ge 1$ be a sequence of Radon measures whose Stieltjes transforms $f_n$, $n\ge 1$ satisfy \eqref{Eq:Radon} and converge pointwise on $\C_+$ to some function $f$. Assume further that the following limit exists
	$$ \lim_n \int_\R \frac1{1+ \lambda^2 } \mu_n(d\lambda) =: b' \in \R\;.$$
	Then the sequence $\mu_n$ converges vaguely to a Radon measure $\mu$ that satisfies
	$$ \int_\R \frac1{1+ \lambda^2 } \mu(d\lambda) < \infty\;,$$
	and which is related to $f$ through
	\begin{equation}\label{Eq:f}
	f(z) = bz + a + \int_\R \frac{1+z\lambda}{\lambda - z} \frac{\mu(d\lambda)}{1+\lambda^2}\;,
	\end{equation}
	with $a := \Re(f(i))$ and $b:=b' - \int (1+\lambda^2)^{-1} \mu(d\lambda)$.
\end{lemma}
The functions of the form \eqref{Eq:f} are called Herglotz-Nevanlinna functions, see for instance~\cite[Sec. 3.4]{Teschl} or~\cite[Appendix B]{AizenmanWarzel}. We did not find the convergence result of Lemma \ref{Lemma:Stieltjes} in the literature so we provide a brief proof of it.
\begin{proof}
	Set
	$$ b_n' := \int_\R \frac1{1+ \lambda^2 } \mu_n(d\lambda)\;,\quad a_n := \Re(f_n(i)) = \int_\R \frac{\lambda}{1+\lambda^2} \mu_n(d\lambda)\;,\quad \nu_n(d\lambda) = \frac{\mu_n(d\lambda)}{1+ \lambda^2}\;.$$
	We can rewrite $f_n$ in the following form
	\begin{equation}\label{Eq:f_n}
		f_n(z) = b_n' z + a_n + (1+z^2) \int_\R \frac1{\lambda - z} \nu_n(d\lambda)\;,
	\end{equation}
	By assumption, the measure $\nu_n$ has a finite total mass which is uniformly bounded over $n\ge 1$. By Helly's selection theorem (which relies on a diagonal extraction from the distribution functions at rationals points), there exists a subsequence $\nu_{n_k}$, $k\ge 1$ that converges vaguely to some limit $\nu$ which is a finite measure. We thus set $\mu(d\lambda) := (1+ \lambda^2) \nu(d\lambda)$. By hypothesis, $a_n$ and $b_n'$ converge to $a$ and $b'$ as $n\to\infty$. Moreover, the vague convergence of $\nu_{n_k}$ towards $\nu$ (and the boundedness of their total-masses) implies the convergence of their Stieltjes transforms. Passing to the limit on \eqref{Eq:f_n} along the subsequence $n_k$ we obtain
	$$ f(z) = b'z + a + (1+z^2) \int_\R \frac1{\lambda - z} \nu(d\lambda)\;.$$
	Since the Stieltjes transform characterizes finite measures, we deduce from this identity that $\nu$ is completely characterized by $f$, $a$ and $b'$. The latter do not depend on the chosen subsequence $n_k$, and consequently $\nu_n$ converges vaguely to $\nu$ as $n\to\infty$. This immediately implies that $\mu_n$ converges vaguely to $\mu$. Finally, by adding and subtracting $z\int (1+\lambda^2)^{-1} \mu(d\lambda)$ in the last identity, we easily obtain \eqref{Eq:f}.
\end{proof}
\begin{remark}
	 One can interpret $b$ as the part of the mass of $\nu_n$ that escapes to infinity in the limit $n\to\infty$. If $\nu_n$ happens to converge weakly to $\nu$, then $b=0$. If in addition $\int_\R \vert \lambda \vert \nu(d\lambda) < \infty$ then $f$ is nothing but the Stieltjes transform of $\mu$
	 $$ f(z) = \int_\R \frac{1}{\lambda - z} \mu(d\lambda)\;.$$
\end{remark}
We now proceed with the proof of the convergence of the matrix spectral measure.
\begin{proof}[Proof of Lemma \ref{Lemma:M}]
	We start with the convergence of $\xi_L$. We aim at applying Lemma \ref{Lemma:Stieltjes}, but we need to check that the required hypotheses are satisfied. To that end, we first collect a few properties on the kernel $G_L(z,s,t)$ of $(\cH_L-z)^{-1}$ for $z\in \C_+$. We already know from Proposition \ref{Prop:Circle} that this function is jointly continuous in $s,t \in [-L/2,L/2]$ so that $(\cH_L-z)^{-1}$ is Hilbert-Schmidt and therefore
	$$ \sum_{k} \| (\cH_L-z)^{-1} \varphi_{k,L}\|^2 = \sum_{k} \vert \lambda_{k,L}-z\vert^{-2} < \infty\;.$$
	We thus deduce that the kernel admits the following representation
	\begin{equation}\label{Eq:GL}
		G_L(z,s,t) = \sum_{k\ge 1} (\lambda_{k,L}-z)^{-1} \varphi_{k,L}(s)\varphi_{k,L}(t)\;,
	\end{equation}
	where the series converges in $L^2((-L/2,L/2)^2)$. We examine separately the real and imaginary parts of this series. The latter is given by
	$$ \Im(G_L(z,s,t)) = \sum_{k\ge 1} \frac{\Im(z)}{\vert\lambda_{k,L}-z\vert^2} \varphi_{k,L}(s)\varphi_{k,L}(t)\;,$$
	and is clearly of the form \eqref{dec_kernel_Mercer}. Mercer's Theorem then implies that the series converges absolutely at every point and uniformly over $[-L/2,L/2]^2$. Regarding the real part, it is given by
	$$ \Re(G_L(z,s,t)) = \sum_{k\ge 1} \frac{\lambda_{k,L}-\Re(z)}{\vert\lambda_{k,L}-z\vert^2} \varphi_{k,L}(s)\varphi_{k,L}(t)\;.$$
	This kernel is not necessarily non-negative and therefore does not decompose as in \eqref{dec_kernel_Mercer}. However, since $\lambda_{k,L}\to +\infty$ as $k\to\infty$, there exists $n=n(L,z) \ge 1$ such that for all $k\ge n$, $\lambda_{k,L} \ge \Re(z)$ and therefore the series can be split into a \emph{finite} sum of negative continuous kernels and a series which is of the form \eqref{dec_kernel_Mercer}. Mercer's Theorem thus applies to this second term and this ensures that the whole series converges absolutely at every point and uniformly over $[-L/2,L/2]^2$. As a consequence the series in \eqref{Eq:GL} satisfies the same property. In particular
	$$ G_L(z,0,0) = \sum_{k\ge 1} (\lambda_{k,L}-z)^{-1}\varphi_{k,L}(0)^2 \;,$$
	converges absolutely. Consequently $\xi_L$ satisfies the integrability condition \eqref{Eq:Radon}, thus admits a Stieltjes transform that coincides with $G_L(z,0,0)$. It remains to check that $G_L(z,0,0)$ and $\int_\R (1+\lambda^2)^{-1} \xi_L(d\lambda)$ converge.\\
	Recall from Proposition \ref{Prop:Circle} that $G_L(z,s,t)$ can be expressed in terms of two solutions $y^{L/2}$ and $y^{-L/2}$ of $-y''+\xi y = zy$ satisfying the Dirichlet b.c.~at $+L/2$ and $-L/2$ respectively. The convergence of the Weyl's functions $m_z(\pm L/2,0)$ towards $m_z(\pm \infty)$ discussed in Subsection \ref{Subsec:Weyl} then ensures that $G_L(z,s,t)$ converges pointwise to $G(z,s,t)$, the kernel of $(\cH-z)^{-1}$. In particular, $G_L(z,0,0)$ converges pointwise to $G(z,0,0)$ for any $z\in \C_+$.\\
	Let us now check that for any $z\in \C_+$
	\begin{equation}\label{Eq:xiGz}
		\lim_L \int_\R \Big\vert \frac1{ \lambda -z } \Big\vert^2 \xi_L(d\lambda) =  \int_{\R} \vert G(z,0,t)\vert^2 dt\;.
	\end{equation}
	First of all, the uniform convergence of the series in \eqref{Eq:GL} together with the fact that $(\varphi_{k,L})_k$ forms an orthogonal basis of $L^2([-L/2,L/2])$ yields
	$$ \int_{-L/2}^{L/2} \vert G_L(z,0,t)\vert^2 dt = \sum_k \Big\vert \frac1{\lambda_{k,L} - z} \Big\vert^2 \varphi_{k,L}(0)^2\;.$$
	Second, we claim that the l.h.s.~converges as $L\to\infty$ to $\int_{\R} \vert G(z,0,t)\vert^2 dt$. From the explicit expressions of these kernels and the convergence of the Weyl's functions, it boils down to proving that
	$$ \int_{0}^{L/2} \vert y^{L/2} \vert^2 dt \to \int_{0}^{+\infty} \vert y^{+} \vert^2 dt\;,$$
	and
	$$ \int_{-L/2}^0 \vert y^{-L/2} \vert^2 dt \to \int_{-\infty}^0 \vert y^{-} \vert^2 dt\;.$$
	We concentrate on the first convergence. Note that $\int_{L/2}^{+\infty} \vert y^{+} \vert^2 dt \to 0$ as $L\to\infty$. Furthermore
	\begin{equation}\label{Eq:mzmz}
		\int_{0}^{L/2} \vert y^{L/2} - y^+ \vert^2 dt = \vert m_z(L/2,0) - m_z(+\infty) \vert^2  \int_{0}^{L/2} \vert y^{D} \vert^2 dt\;.
	\end{equation}
	Recall that $m_z(L/2,0)$ and $m_z(+\infty)$ belong to a disk of radius $(2\vert \Im(z)\vert \int_{0}^{L/2} \vert y^{D} \vert^2 dt)^{-1}$ and that this radius vanishes as $L\to\infty$. Consequently $\vert m_z(L/2,0) - m_z(+\infty) \vert$ is smaller than twice this radius, and we can deduce that \eqref{Eq:mzmz} vanishes as $L\to\infty$. We have thus proved \eqref{Eq:xiGz}. Specializing this identity to $z=i$, we deduce the convergence of $\int_\R (1+\lambda^2)^{-1} \xi_L(d\lambda)$.
	
	We can now apply Lemma \ref{Lemma:Stieltjes} to deduce the vague convergence of $\xi_L$ towards some Radon measure $\xi$, and that the following identity holds (for some parameters $a,b\in\R$)
	\begin{equation}\label{Eq:Gxi}
		G(z,0,0) = bz + a + \int_\R \frac{1+z\lambda}{\lambda - z} \frac{\xi(d\lambda)}{1+\lambda^2}\;.
	\end{equation}
	
	The arguments are similar for $\zeta_L$ and $\eta_L$ but some additional difficulties arise. Let us concentrate on $\zeta_L$. First of all, note that $G_L(z,s,t)$ is $\cC^1$ outside the diagonal $\{(s,t) \in [-L/2,L/2]:s=t\}$ and that $\partial_s\partial_t G_L(z,s,t)$ can be extended continuously up to the diagonal. This function can be seen as the kernel of a Hilbert-Schmidt operator, that we denote $\partial_s\partial_t (\cH_L-z)^{-1}$. Note that for all $g,h \in H^1_0(-L/2,L/2)$
	$$ \langle g, \partial_s \partial_t (\cH_L-z)^{-1} h\rangle = \sum_{k} (\lambda_{k,L}-z)^{-1} \langle g, \varphi_{k,L}'\rangle \langle h, \varphi_{k,L}'\rangle\;.$$
	Now observe that $\Im(\partial_s\partial_t (\cH_L-z)^{-1})$ is a non-negative operator on $H^1_0$, and therefore by density also on $L^2$. Mercer's Theorem then yields that
	$$ \Im(\partial_s \partial_t G_L(z,s,t)) = \sum_{k\ge 1} \Im((\lambda_{k,L}-z)^{-1}) \varphi_{k,L}'(s) \varphi_{k,L}'(t)\;,$$
	where the convergence is absolute at every point and uniform over $[-L/2,L/2]^2$. A similar argument applies to the real part (again, one needs to take care of a finite number of negative terms as above). We thus deduce that
	$$ \partial_s \partial_t G_L(z,s,t) = \sum_{k\ge 1} (\lambda_{k,L}-z)^{-1} \varphi_{k,L}'(s) \varphi_{k,L}'(t)\;,$$
	where the convergence is absolute pointwise and uniform over $[-L/2,L/2]^2$. In particular
	$$ \partial_s \partial_t G_L(z,0,0) = \sum_{k\ge 1} (\lambda_{k,L}-z)^{-1} \varphi_{k,L}'(0)^2 = \int (\lambda-z)^{-1} \zeta_L(d\lambda)\;.$$
	Here again, the convergence of the Weyl's functions $m_z(\pm L/2,0)$ towards $m_z(\pm\infty)$ discussed in Subsection \ref{Subsec:Weyl} and the explicit expressions of the kernels at stake show that
	$$ \lim_L\partial_s \partial_t G_L(z,0,0) = \partial_s \partial_t G(z,0,0)\;,\quad \lim_L \int_\R \Big\vert \frac1{\lambda-z } \Big\vert^2 \zeta_L(d\lambda) =  \int_{\R} \vert \partial_x G(z,0,t)\vert^2 dt\;.$$
	and we deduce from Lemma \ref{Lemma:Stieltjes} that $\zeta_L$ converges vaguely to some Radon measure $\zeta$.
\end{proof}

Let us conclude this subsection with a comment on the matrix $M$. Once we know that $\cH$ has pure point spectrum, with eigenvalues/eigenfunctions denoted $(\lambda_k,\varphi_k)_{k\ge 1}$, one would like to show the very natural identities:
$$ \xi := \sum_{k\ge 1} \varphi_k(0)^2 \delta_{\gl_k}\;,\quad \zeta := \sum_{k\ge 1} \varphi_k'(0)^2 \delta_{\gl_k}\;,\quad \eta := \sum_{k\ge 1} \varphi_k(0)\varphi_k'(0) \delta_{\gl_k}\;.$$
Let us show the first identity. Take $z\in \C_+$. The spectral theorem ensures that for all $f,g \in L^2(\R)$
$$ \langle f, (\cH-z)^{-1} g\rangle = \sum_{k\ge 1}(\lambda_k - z)^{-1} \langle f,\varphi_k\rangle \langle g,\varphi_k\rangle\;,$$
where the series converges absolutely. By Proposition \ref{Prop:H}, the operator $\Im((\cH-z)^{-1})$ is a positive operator with a jointly continuous kernel $\Im(G(z,s,t))$. By Mercer's Theorem, we deduce that
$$ \sum_{k\ge 1} \frac{\Im(z)}{\vert \lambda_k -z \vert^2}  \varphi_k(s) \varphi_k(t)\;,$$
converges absolutely pointwise and locally uniformly over $\R^2$. This implies that
$$ \Im(G(z,0,0)) = \sum_{k\ge 1} \frac{\Im(z)}{\vert \lambda_k -z \vert^2}  \varphi_k(0)^2\;,$$
where the convergence is absolute. Recall from \eqref{Eq:Gxi} that $G(z,0,0)$ is a Herglotz-Nevanlinna function with parameters to $a$, $b$ and $\xi$. It is a well-known fact on Herglotz-Nevanlinna functions that $\xi$ can be read off the imaginary part of $G(z,0,0)$, see for instance~\cite[Prop. B.1]{AizenmanWarzel}. In particular, for any $c\in\R$
$$ \xi(\{c\}) = \lim_{\varepsilon\downarrow 0} \varepsilon  \Im(G(c+i\epsilon,0,0) )\;.$$
We already know that $\xi$ is pure point, so this last identity allows to deduce that $\xi := \sum_{k\ge 1} \varphi_k(0)^2 \delta_{\gl_k}$.

\subsection{Proof of Proposition \ref{Prop:Expo}}{\label{Subsec:Diffusions}}

We need to collect a few properties on the diffusion $\theta_\gl$ that were established in~\cite{DL_Crossover}. The process $\theta_\gl$ mod$[\pi]$ is Markovian and admits a unique invariant measure $\mu_\gl$. The explicit integral expression satisfied by the density (also denoted $\mu_\gl$) will be unnecessary in the present work, however we will need to know that for any compact set $\Delta \subset \R$
\begin{equation}\label{Eq:mu}
	\sup_{\lambda \in \Delta} \sup_{\theta \in [0,\pi)} \mu_\gl(\theta) < \infty\;.
\end{equation}

The transition probabilities $p_{\gl,t}(\theta_0,\theta)$ of the diffusion $\theta_\gl$ mod$[\pi]$ starting from $\theta_\gl(0) = \theta_0$ converge exponentially fast to the invariant measure and we have
\begin{equation}\label{Eq:pmu}
	\sup_{\lambda \in \Delta} \sup_{\theta_0,\theta \in [0,\pi)} \vert p_{\gl,t}(\theta_0,\theta) - \mu_\gl(\theta) \vert \to 0\;,\quad t\to\infty\;.
\end{equation}
A trivial consequence of this convergence is that the transition probabilities are uniformly bounded: for any $t_0 >0$ we have
\begin{equation}\label{Eq:plambda}
	\sup_{\gl\in\Delta} \sup_{t\ge t_0} \sup_{\theta\in [0,\pi)}p_{\gl,t}(\theta) < \infty\;.
\end{equation}

We already mentioned that the solution $\rho_\gl$ of \eqref{EDSlnr} is completely determined by the trajectory of $\theta_\gl$ (and the initial condition $\rho_\gl(0)$). We denote by $\P_{(t_0,\theta_0)}$ the law of $\theta_\gl$ (or $(\theta_\gl,\rho_\gl)$) starting at time $t_0$ from $\theta_0$. We also let $\P_{(t_0,\theta_0) \to (t,\theta)}$ be the law of the same diffusion conditioned on hitting $\theta+\pi\Z$ at time $t$: this is a \emph{bridge} of diffusion.\\

Let us now introduce the so-called concatenation process, which is instrumental in the study of the eigenproblem associated to $\cH_L$. Let $(\theta^+_\gl,\rho^+_\gl)$, resp.~$(\theta^-_\gl,\rho^-_\gl)$, be distributed according to $\P_{(-L/2,0) \to (0,\theta)}$, resp.~$\P_{(-L/2,0) \to (0,\pi-\theta)}$ with $\rho_\gl^+(0) = \rho_\gl^-(0) = 0$ (note that we do not impose an initial condition on $\rho_\gl^\pm$ but rather a terminal condition: this is not a problem since $\rho_\gl$ satisfies an additive equation). We then consider the concatenation of these two independent processes by setting
$$ (\hat{\rho}_\gl(t),\hat{\theta}_\gl(t)) = \begin{cases}(\rho^+_\gl(t),\theta^+_\gl(t))&\mbox{ if }t\in [-L/2,0]\\
(\rho^-_\gl(-t),k\pi - \theta^-_\gl(-t))&\mbox{ if }t\in (0,L/2]\;,
\end{cases}$$
where $k := \lfloor (\theta_\gl^+(0) + \theta_\gl^-(0))/\pi\rfloor$. We denote by $\P^{(0)}_{\theta,\pi-\theta}$ the corresponding law. We can naturally define
$$ \hat{y}_\gl(t) = \hat{r}_\gl(t) \sin(\hat{\theta}_\gl(t))\;.$$

\begin{proof}[Proof of Proposition \ref{Prop:Expo}]

Without loss of generality we can assume that the bounded interval $\Delta$ is open. Let us introduce
$$ G_A(\gl) := \inf_{\theta_0\in [0,\pi)} \sup_{t\in (-A/2,A/2)} \Big(r_{\gl,\theta_0}(t) e^{(\gamma_\gl-\epsilon)|t|}  + \frac1{r_{\gl,\theta_0}(t)} e^{-(\gamma_\gl+\epsilon)|t|} \Big)^q\;,\quad A\in (0,\infty]\;.$$
Almost surely for every $\gl\in\R$ we have
$$ G_\infty(\gl)  = \lim_{A\to\infty} \uparrow G_A(\gl)\;.$$
Indeed, $G_\infty(\gl)  \ge \lim_{A\to\infty} \uparrow G_A(\gl)$ is immediate. Now by contradiction, assume that the inequality is strict: then there exists $\varepsilon>0$ and a sequence $(\theta_n,A_n)$ such that $\theta_n\to\theta_0$, $A_n\uparrow \infty$ and $G_{A_n}(\gl,\theta_n)  < G_\infty(\gl) - \varepsilon$ where
$$ G_A(\gl,\theta)   := \sup_{t\in (-A/2,A/2)} \Big(r_{\gl,\theta}(t) e^{(\gamma_\gl-\epsilon)|t|}  + \frac1{r_{\gl,\theta}(t)} e^{-(\gamma_\gl+\epsilon)|t|} \Big)^q \;.$$
By continuity, we know that $G_A(\gl,\theta_0) = \lim_n G_A(\gl,\theta_n)$, and since $A\mapsto G_A(\gl,\theta)$ is non-decreasing we deduce that for all $n\ge 1$ such that $A_n>A$
$$ G_A(\gl,\theta_n) \le G_{A_n}(\gl,\theta_n) < G_\infty(\gl) - \varepsilon\;,$$
thus implying that $G_\infty(\gl,\theta_0) < G_\infty(\gl) - \varepsilon$, a contradiction.\\

Consequently, the Monotone Convergence Theorem ensures that it suffices to prove the bound
$$ \sup_{A>0} \E\Big[ \int_\Delta G_A(\gl) \sigma(d\gl)\Big] < \infty\;.$$

Fix $A>0$. The almost sure vague convergence of $\sigma_L$ to $\sigma$, together with the almost sure continuity of $\gl\mapsto G_A(\gl)$, shows that a.s.
$$ \int_\Delta G_A(\gl) \sigma(d\gl) \le \liminf_{L\to\infty} \int_\Delta G_A(\gl) \sigma_L(d\gl)\;.$$
Applying Fatou's Lemma, we obtain
$$ \E\Big[ \int_\Delta G_A(\gl) \sigma(d\gl)\Big] \le \liminf_{L\to\infty}  \E\Big[ \int_\Delta G_A(\gl) \sigma_L(d\gl)\Big]\;.$$
We are therefore left with proving the bound
\begin{equation}\label{Eq:ExpoToProve} \sup_{A>0} \sup_{L>A}  \E\Big[ \int_\Delta G_A(\gl) \sigma_L(d\gl)\Big] < \infty\;.\end{equation}

Recall the expression of the spectral measure $\sigma_L$ and observe that
$$ \E\Big[ \int_\Delta G_A(\gl) \sigma_L(d\gl)\Big] = \E\Big[\sum_{k\ge 1} \big(\varphi_{k,L}(0)^2 + \varphi_{k,L}'(0)^2\big)\;G_A(\gl_{k,L},\varphi_{k,L},\varphi_{k,L}')\Big]\;,$$
where
$$ G_A(\gl,f,g) = \mathbf{1}_\Delta(\gl) \sup_{t\in (-A/2,A/2)} \Big(\frac{\sqrt{f^2(t) + g^2(t)}}{\sqrt{f^2(0) + g^2(0)}} e^{(\gamma_\gl-\epsilon)|t|}  + \frac{\sqrt{f^2(0) + g^2(0)}}{\sqrt{f^2(t) + g^2(t)}} e^{-(\gamma_\gl+\epsilon)|t|} \Big)^q\;.$$
Then, we rely on the following formula established in~\cite[Prop 6.1]{DL_Crossover} (take $u=0$ and $\mathbf{E} = 1$ therein):
\begin{align*}
&\E\big[\sum_{i\ge 1} \big(\varphi_{i,L}(0)^2 + \varphi_{i,L}'(0)^2\big)\;G_A(\gl_{i,L},\varphi_{i,L},\varphi_{i,L}')\big] \\
&=\int_{\gl\in\R} \int_{\theta=0}^\pi p_{\gl,\frac{L}{2}}(\theta) p_{\gl,\frac{L}{2}}(\pi-\theta)
 \E^{(0)}_{\theta,\pi-\theta}\Big[G_A\Big(\gl,\frac{\yu_\gl}{\|\yu_\gl\|},\frac{\yu_\gl'}{\|\yu_\gl\|}\Big)\Big] \, d\theta d\gl\;.
\end{align*}
Since the transition probabilities are uniformly bounded, see \eqref{Eq:plambda}, it suffices to prove
$$ \sup_{A>0}\sup_{L>A}\sup_{\gl\in\Delta } \sup_{\theta\in [0,\pi)}\E^{(0)}_{\theta,\pi-\theta}\Big[G_A\Big(\gl,\frac{\yu_\gl}{\|\yu_\gl\|},\frac{\yu_\gl'}{\|\yu_\gl\|}\Big)\Big] <\infty\;,$$
which can be rewritten as (recall that $\hat{r}_\gl(0) = 1$)
$$ \sup_{A>0}\sup_{L>A}\sup_{\gl\in\Delta} \sup_{\theta\in [0,\pi)}\E^{(0)}_{\theta,\pi-\theta}\Big[\sup_{t\in (-A/2,A/2)} \Big(\ru_\gl(t) e^{(\gamma_\gl-\epsilon)|t|}  + \frac1{\ru_\gl(t)} e^{-(\gamma_\gl+\epsilon)|t|} \Big)^q \Big] < \infty\;.$$
Here we deal with the concatenation of two independent bridges on $(-L/2,0)$ and $(0,L/2)$. By symmetry, it suffices to prove (recall that $\rho_\gl(0) = 1$)
$$ \sup_{A>0}\sup_{L>A}\sup_{\gl\in\Delta} \sup_{\theta\in [0,\pi)}\E_{(-L/2,0)\rightarrow (0,\theta)}\Big[\sup_{t\in [-A/2,0]} \Big(r_\gl(t) e^{(\gamma_\gl-\epsilon)|t|}  + \frac1{r_\gl(t)} e^{-(\gamma_\gl+\epsilon)|t|} \Big)^q \Big] < \infty\;.$$
This can be rewritten in the more convenient form
$$ \sup_{A>0}\sup_{L>A}\sup_{\gl\in\Delta} \sup_{\theta\in [0,\pi)}\E_{(0,0)\rightarrow (L/2,\theta)}\Big[\sup_{t\in [(L-A)/2,L/2]} e^{\frac{q}2\vert \rho_\gl(t) - \rho_\gl(L/2) - 2\gamma_\gl(t-L/2)\vert -q\epsilon|t-L/2|}\Big] < \infty\;.$$

We have therefore reduced the problem to some exponential moments on the bridge of the solution to an additive SDE. To that end, we will consider the time-reversal of this bridge (this is more convenient since the estimate concerns the final portion of the bridge) and we will be able to disregard the bridge condition by showing that the bridge is absolutely continuous w.r.t.~the unconditioned diffusion as long as we are concerned with its behavior not too close to the end point.\\

We rely on the adjoint diffusions $(\bar{\theta}_\gl,\bar{\rho}_\gl)$ defined in~\cite[Sect 7.2]{DL_Crossover}: for any Borel set $F$ of the set of continuous functions on $[0,t]$ we have
$$ \mu_\gl(\theta_0) p_{\gl,t}(\theta_0,\theta_1) \P_{(0,\theta_0) \to (t,\theta_1)}(F) =  \mu_\gl(\theta_1) \bar{p}_{\gl,t}(\theta_1,\theta_0) \bar{\P}_{(0,\theta_1) \to (t,\theta_0)}(\bar{F})\;,$$
where $\bar{F}$ is the image of $F$ upon reversing time. Since $p_{\gl,t}$ and $\bar{p}_{\gl,t}$ converge to $\mu_\gl$ uniformly over all $\theta_0$ and all $\gl$ in a compact set, it suffices to show that
$$ \sup_{A>0} \sup_{L>A}\sup_{\gl\in\Delta} \sup_{\theta\in [0,\pi)}\bar{\E}_{(0,\theta)\rightarrow (L/2,0)}\Big[\sup_{t\in [0,A/2]} e^{\frac{q}2\vert \bar{\rho}_\gl(t) + 2\gamma_\gl t\vert -q\epsilon t} \Big] < \infty\;.$$
By the uniform absolute continuity stated in~\cite[Lemma 7.3]{DL_Crossover}, we can replace $\bar{\E}_{(0,\theta)\rightarrow (L/2,0)}$ by $\bar{\E}_{(0,\theta)}$. The existence of a parameter $q>0$ such that the above quantity is indeed finite is then a consequence of~\cite[Lemma 7.2]{DL_Crossover} (and of the fact that $2\gamma_\gl = \nu_\gl$).\\
Note that the two lemmas that we have just quoted state uniform estimates on a microscopic interval in $\lambda$ (the size of $\Delta$ therein is of order $1/L$) however the proofs of these lemmas rely on previous estimates in~\cite{DL_Crossover} that hold uniformly over macroscopic intervals in $\lambda$.
\end{proof}

\begin{remark}[Perturbation of the noise]
	The main technical input in our proof of localization is actually the only sensitive part to a perturbation of the noise: Proposition \ref{Prop:Expo}. Indeed, while all other arguments still work with $\xi_\varepsilon$, the proof of this proposition relies on estimates on our diffusions that would not be easy to obtain if we lose the Markov property induced by white noise.
\end{remark}

\section{Connection with the PAM}\label{Sec:PAM}

In this section, we collect spectral information on $\cH$ using mostly PDE arguments.  More precisely, we exploit the connection between the parabolic Anderson model, whose construction can be carried out with standard PDE arguments, and the semigroup associated to the operator $\cH$, to deduce that the spectrum of $\cH$ equals $\R$ and to establish a priori growth estimates on generalized eigenfunctions (that will be required in Section \ref{Sec:Kotani} to follow the Kotani-Simon approach for proving Anderson localization). Let us point out that the singularity of $\xi$ induces some difficulties in establishing the latter growth estimates.

\subsection{The PAM and its connection to the hamiltonian}\label{Subsec:PAM}

We consider the parabolic Anderson model (PAM) defined as the solution of the following ODE
$$ \begin{cases}
	\partial_t u &= \partial_x^2 u - \xi u\;, \mbox{ on }(0,\infty)\times \R\;,\\
	u(0,x) &=u_0(x)\;,
\end{cases}$$
for some initial condition $u_0$. More precisely, let $P_t(x)$ be the heat kernel, we look for a function $u:(0,\infty)\times\R \to \R$ that satisfies\footnote{The integrals w.r.t.~$y$ have to be understood in the sense of distributions.}
\begin{equation}\label{Eq:PAMmild}
	u(t,x) = \int_y P_t(x-y) u_0(y) dy - \int_0^t \int_y P_{t-s}(x-y) \xi(dy) u(s,y) ds\;.
\end{equation}
The existence and uniqueness of the solutions of this PDE belongs to the folklore of the literature on the PAM. For completeness, we give in Proposition \ref{Prop:Weight} a precise existence and uniqueness result for this PDE. To state it, we need to introduce weighted Sobolev spaces and a couple of parameters that control regularity and integrability properties of the functions at stake.\\

Let us mention that there are two sources of difficulties when solving the above PDE. One comes from the irregularity of the objects at stake: white noise is distribution-valued and one therefore needs to be careful when dealing with products of functions/distributions. Another one concerns the lack of integrability at $\pm\infty$ of the objects at stake due to the unboundedness of the spatial domain on which we solve the PDE: indeed, white noise, as a distribution on $\R$, only lives in \emph{weighted} H\"older spaces of distributions and this implies that the whole solution theory needs to be set up in weighted functions/distributions spaces.\\

We denote by $H^\gamma$ the classical Sobolev space of regularity index $\gamma \in \R$. For a so-called weight function $w:\R\to \R_+$, let $H^\gamma_w$ be the set of distributions $f$ whose local Sobolev norm does not grow faster than $w$, more precisely:
$$ \| f \|_{H^\gamma_w}^2 := \sum_{k\in\Z} \Big(\frac{\| f(\cdot) \chi(\cdot-k)\|_{H^\gamma}}{w(k)} \Big)^2 < \infty\;,$$
where $\chi$ is an arbitrary smooth function, supported in $[-2,2]$, that equals $1$ on $[-1,1]$.\\

Below we will deal with the two-parameter weight function
$$ w_{\beta,\ell}(x) := \exp\big( \ell (1+\vert x \vert)^{\beta} \big)\;,\quad x\in\R\;,$$
where $\beta \in (0,1]$ and $\ell\in\R$. Observe that $\ell > 0$ allows for a growth at infinity, while $\ell < 0$ implies a decay. Note that the Dirac mass $\delta_{x}$ belongs to $H^{-1/2}$ and, since it is compactly support, it belongs to $H^{-1/2}_{w_{\beta,\ell}}$ for any $\ell \in \R$ and any $\beta \in (0,1]$.

Given some  parameters $q\in (0,1)$ and $T>0$, we define the Banach space $\cE=\cE(\gamma,q,T,\ell,\beta)$ as the closure of all compactly supported, smooth functions $u:(0,T)\times\R\to\R$ under the norm
$$ \|u\|_{\cE} := \sup_{t\in (0,T]} t^q \| u(t,\cdot) \|_{H^{\gamma}_{w_{\beta,\ell+t}}} < \infty\;.$$
Note that the parameter $q$ allows for a blow-up of $u$ at time $0$: this is useful to start the PAM from an initial condition with low regularity.

\begin{proposition}\label{Prop:Weight}
	Fix $\beta \in  (0,1]$, $\ell \in \R$, $\alpha > -3/2$ and $T>0$, and take $\gamma < (\alpha +2) \wedge 3/2$. There exists $q \in (0,1)$ such that for any $u_0 \in H^\alpha_{w_{\beta,\ell}}$, there exists a unique solution $u\in \cE$ of \eqref{Eq:PAMmild} starting from $u_0$. In addition, the solution map $u_0 \mapsto u$ is Lipschitz in the above norms. Finally, the solution $u$ satisfies in the sense of distributions
	\begin{equation}\label{Eq:PDE}
		\partial_t u = \partial^2_x u - \xi u\;,\quad \mbox{ on }(0,\infty) \times \R\;.
	\end{equation}
\end{proposition}	

A simple consequence of this result combined with standard Sobolev Embeddings is that, at any time $t>0$, there exists $C=C(t)>0$ such that the unique solution $u$ of the PAM starting from $u_0 \in H^\alpha_{w_{\beta,\ell}}$ satisfies the bound
$$  \sup_{x\in\R} \frac{\vert u(t,x)\vert}{w_{\beta,\ell+t}(x)} \le C \| u_0 \|_{H^\alpha_{w_{\beta,\ell}}} \;.$$
Note that the growth of the solution at infinity is comparable to the growth of the initial condition: while the index $\beta$ remains the same, the prefactor $\ell$ is increased by $t$.\\


Let us make a comment on \eqref{Eq:PDE}. First of all, we preferred to write $\partial^2_x u - \xi u$ on the r.h.s.~rather than $-\cH u$ since $u(t,\cdot)$ does not necessarily belong to the domain of $\cH$. Second, this equation makes sense provided the product $\xi u(t,\cdot)$ is well-defined: this is the case here since $u(t,\cdot)$ belongs locally to $H^{1/2+\kappa}$ for some $\kappa>0$.

\begin{proof}[Sketch of proof of Proposition \ref{Prop:Weight}]
	The proof follows essentially the same steps as in~\cite[Sec 4]{PAM2D} where the solution theory of the PAM in dimension $2$ is established. Let us briefly present the main arguments. The core of the proof is to show that, for $T> 0$ small enough, the map $u\mapsto \cM(u)$ is contractive from $\cE$ to $\cE$ where
	$$ \cM(u)(t,x) := \int_y P_t(x-y) u_0(y) dy - \int_0^t \int_y P_{t-s}(x-y) \xi(dy) u(s,y) ds\;.$$
	The main difficulty consists in evaluating the $\cE$-norm of the second term. To that end, one first observes that for any given $\kappa > 0$ and $a>0$, $\xi$ almost surely belongs to the weighted H\"older space $\cC^{-1/2-\kappa}_p$ where $p(x):= (1+|x|)^a$. Second there exist $C=C(a/\beta) > 0$ such that for all $s < t$,
	$$ \sup_{x\in\R} \frac{p(x) w_{\beta,\ell+s}(x)}{w_{\beta,\ell+t}(x)} \leq C (t-s)^{-\frac{a}{\beta}}\;.$$
	Consequently at any time $s\in (0,t)$
	$$ \Big\| \int_y P_{t-s}(\cdot-y) v(y) \Big\|_{H^{\gamma}_{w_{\beta,\ell+t}}} \le C (t-s)^{-\frac{a}{\beta}} \Big\| \int_y P_{t-s}(\cdot-y) v(y) \Big\|_{H^{\gamma}_{pw_{\beta,\ell+s}}}\;.$$
	Third, $P_{t-s}$ improves Sobolev regularity by $c$ at the price of a prefactor $(t-s)^{-\frac{c}{2}}$, that is, there exists a constant $C>0$ such that
	\[
	\Big\| \int_y P_{t-s}(\cdot-y) v(y) \Big\|_{H^{\gamma}_{pw_{\beta,\ell+s}}} \le C  (t-s)^{-\frac{c}{2}}	\| v \|_{H^{\gamma-c}_{pw_{\beta,\ell+s}}} \;.
	\]
	Fourth, Young's integral ensures that for all $\gamma > 1/2 + \kappa$, there exists a constant $C>0$ such that
	\[
	\big\| \xi u(s,\cdot) \big\|_{H^{-\frac12 - \kappa}_{pw_{\beta,\ell+s}}} \leq C \|\xi\|_{\cC^{-\frac12-\kappa}_p} \|u(s,\cdot)\|_{H^{\gamma}_{w_{\beta,\ell+s}}}\;.
	\]
	Putting everything together, we thus deduce that
	\[
	\Big\| \int_0^t \int_y P_{t-s}(x-y) \xi(dy) u(s,y) ds \Big\|_{H^{\gamma}_{w_{\beta,\ell+t}}} \le C \int_0^t (t-s)^{-\frac{\gamma+\frac12+\kappa}{2}-\frac{a}{\beta}} s^{-q} ds \;  \|\xi\|_{\cC^{-1/2-\kappa}_p} \|u\|_{\cE}\;.
	\]
	Provided $\kappa$ and $a$ are small enough, and $q$ is close enough to $1$ the integral on the right converges, and yields a prefactor which is negligible compared to $t^{-q}$, thus providing the required contractivity.
\end{proof}

The next proposition relates the solution of the PAM to the family of operators $e^{-t\cH}$, $t>0$. While such a result is expected, the present situation is not standard. First of all, as we will see later on the spectrum of $\cH$ happens to be unbounded from above \emph{and} below and therefore the domain of $e^{-t\cH}$ \emph{does not} contain the domain of $\cH$. (Recall that the domain of $e^{-t\cH}$ is the set of all functions $f\in L^2$ such that $\int e^{-2t\lambda} \mu_f(d\lambda) < \infty$). Second, the singularity of the noise prevents from multiplying it with arbitrary $L^2$ functions and one needs Young's integral to make sense of some products. For completeness, we thus provide the main steps of the argument.

\begin{proposition}\label{Prop:PAM}
	For any given time $t>0$, the following holds:\begin{enumerate}[label=(\roman*)]
	\item For any $u_0\in L^2$, the solution $u(t,\cdot)$ of the PAM starting from $u_0$ satisfies
	$$ u(t,x) = \int_{y\in \R} u_0(y) u^{(y)}(t,x) dy\;,\quad x\in\R\;,$$
	where $u^{(y)}$ is the solution of the PAM starting from $\delta_y$.
	\item For any $u_0$ in the domain of $e^{-t\cH}$ we have
	$$ (e^{-t\cH} u_0)(\cdot) = u(t,\cdot)\;,$$
	where $u(t,\cdot)$ is the solution of the PAM starting from $u_0$.
	\end{enumerate}
	As a consequence, $e^{-t\cH}$ admits a jointly continuous kernel given by
	\begin{equation}\label{Eq:PAMKernel}
		e^{-t\cH}(y,x) = u^{(y)}(t,x)\;,\quad (x,y) \in\R^2\;.
	\end{equation}
\end{proposition}

This proposition shows that the operator $e^{-t\cH}$ and the solution of the PAM coincide. In particular, $e^{-t\cH}$ admits a unique continuous extension as an operator from the whole space $L^2$ into a weighted $L^2$ space. More generally, it admits a unique continuous extension as an operator from a weighted Sobolev space into another weighted Sobolev space, provided the parameters are chosen according to the solution theory of the PAM presented in Proposition \ref{Prop:Weight}. In the sequel, we will not distinguish $e^{-t\cH}$ from these (unique) extensions.

\begin{remark}
	This is a nice interplay between the PAM, which is constructed by general PDE arguments, and the operator $e^{-t\cH}$, whose definition relies on spectral arguments. Observe that the PDE arguments are powerful in that they allow to extend the set of functions on which $e^{-t\cH}$ acts. On the other hand, observe that for $u_0 \in L^2$, the solution theory for the PAM only ensures that $u(t,\cdot)$ belongs \emph{locally} to $L^2(\R)$ but not necessarily globally as it may grow at infinity, but the connection with $e^{-t\cH}$ asserts that for $u_0$ in the domain of $e^{-t\cH}$, the solution of the PAM actually belongs to $L^2(\R)$ globally.
\end{remark}
\begin{proof}
	Fix $\beta \in (0,1]$. For any given $\ell\in\R$, observe that $y\mapsto \delta_y$ is a continuous map from $\R$ into $H^{-1/2}_{w_{\beta,\ell}}$. This fact combined with Proposition \ref{Prop:Weight} ensures that $y\mapsto u^{(y)}(t,\cdot)$ is a continuous map with values in $H^{\gamma}_{w_{\beta,t+\ell}}$, for any given $\gamma < 3/2$ and any given $\ell\in\R$. (By Sobolev Embeddings, note also that $(x,y)\mapsto u^{(y)}(t,x)$ is jointly continuous). It is then not difficult to check that, for any continuous function $u_0$ with compact support, the function
	$$ x\mapsto \int_y u^{(y)}(t,x) u_0(y) dy\;,$$
	takes values in $H^{\gamma}_{w_{\beta,t+\ell}}$ and solves the PAM starting from $u_0$. We have therefore shown that (i) holds for all continuous functions $u_0$ with compact support. The next step is to extend this identity to the whole set $L^2$.\\
	
	We deduce from Proposition \ref{Prop:Weight} and from Sobolev Embeddings that
	\begin{equation}\label{Eq:utu0}
		\sup_{x} \sup_{u_0 \in L^2} \frac{\vert u(t,x)\vert}{w_{\beta,t}(x) \|u_0\|_{L^2}} < \infty\;.
	\end{equation}
	In particular, for any given $x$, $u_0\mapsto u(t,x)$ is a bounded linear form on $L^2$. By Riesz representation theorem, this linear form admits a representative in $L^2$. From the first part of the proof, combined with the density of the set of continuous functions with compact support in $L^2$, we deduce that this representative is the continuous function $y\mapsto u^{(y)}(t,x)$. This completes the proof of (i).\\
	
	We turn to the proof of (ii). Take $\beta' \in (\beta, 1)$ and set $w(x) := w_{\beta',1}(x)$. From \eqref{Eq:utu0}, we deduce that
	$$ \int_x  \Big(\sup_{u_0 \in L^2} \frac{\vert u(t,x)\vert}{w(x) \|u_0\|_{L^2}}\Big)^{2} dx < \infty\;.$$
	Consequently $u_0\mapsto u(t,\cdot)$ is a bounded operator from $L^2$ into a weigthed $L^2$ space.\\
	Suppose that we can show that $e^{-t\cH} u_0$ coincides with $u(t,\cdot)$ for a set of functions $u_0$ which is dense in the domain of $e^{-t\cH}$ (and therefore is also dense in $L^2$). This is sufficient to deduce that $e^{-t\cH}$ admits a unique continuous extension into an operator from the whole set $L^2$ into a weighted $L^2$ space, and this extension is of course the operator $u_0\mapsto u(t,\cdot)$. In particular, these two operators coincide on the domain of $e^{-t\cH}$ (viewed as an operator from $L^2$ into itself), and in turn, it ensures that the kernel of $e^{-t\cH}$ coincides with the solution of the PAM starting from a Dirac delta, thus yielding \eqref{Eq:PAMKernel}.\\
	
	It remains to prove the coincidence of the two operators for a convenient dense subset. Recall that $E$ is the projection-valued measure associated to $\cH$. To circumvent integrability issues, we will restrict ourselves to all functions $u_0 \in L^2$ such that $E(\R\backslash A) u_0 = 0$, for some arbitrary bounded set $A \subset \R$: such functions $u_0$ then belong to the domains of $\cH$ and $e^{-t\cH}$ for all $t\in \R$.\\
	The following identity holds (the integral on the r.h.s.~can be interpreted as a Bochner integral)
	$$ e^{-t\cH} u_0 = u_0 + \int_0^t -\cH e^{-s\cH} u_0 ds\;.$$
	In particular, $t\mapsto e^{-t\cH} u_0$ is a continuous function from $\R_+$ to $L^2$.\\
	Starting from the previous identity, we would like to obtain a weak form of the PAM. Integrations by parts show that for all compactly supported and smooth function $\psi:\R^2\to\R$
	\begin{align*}
		\langle e^{-t\cH} u_0, \psi(t,\cdot)\rangle &= \langle u_0, \psi(0,\cdot)\rangle + \int_0^t \langle e^{-s\cH} u_0, \partial_s \psi(s,\cdot) + \partial_x^2 \psi(s,\cdot)\rangle ds\\
		&-  \int_0^t \langle e^{-s\cH} u_0, \xi \psi(s,\cdot)\rangle ds\;,
	\end{align*}
	where we used the fact that $e^{-s\cH}u_0$ has H\"older regularity $3/2-$ (since it belongs to the domain of $\cH$), and so its product with $\xi$ is well-defined as a Young's integral.\\
	This identity can be applied to $\psi(s,x) := P_{t_0-s}(x_0-x)$ for some $t_0 > t$ and $x_0 \in\R$, and this yields
	$$ \langle e^{-t\cH} u_0, P_{t_0-t}(x_0-\cdot) \rangle = \langle u_0, P_{t_0}(x_0-\cdot)\rangle - \int_0^t \langle e^{-s\cH} u_0, \xi P_{t_0-s}(x_0-\cdot) \rangle ds\;.$$
	To conclude, we would like to pass to the limit $t_0 \downarrow t$ and obtain the identity
	$$  e^{-t\cH} u_0(x_0) = \langle u_0, P_{t}(x_0-\cdot)\rangle - \int_0^t \langle e^{-s\cH} u_0, \xi P_{t-s}(x_0-\cdot) \rangle ds\;.$$
	This can be done easily on the l.h.s., and on the first term on the r.h.s. Regarding the second term on the r.h.s., this can be checked by a direct estimate that relies on Young's integral.
\end{proof}

\subsection{Growth estimates on the generalized eigenfunctions}\label{Subsec:Growth}

We start with a disintegration result of the projection-valued measure $E(d\lambda)$ into the product of a kernel $E(\lambda,x,y)$ with some spectral measure $\varrho(d\lambda)$. The singularity of white noise makes the proof of this result more delicate, see below for more explanations.

\begin{proposition}\label{Prop:Desintegration}
	There exists a spectral measure $\varrho$, and a jointly measurable function $(\lambda,x,y)\mapsto E(\lambda,x,y)$ such that for all bounded and measurable functions $f,g:\R\to\R$ with compact supports and all bounded measurable set $A\in \cB(\R)$
	$$ \langle f, E(A) g\rangle = \int_\R \mathbf{1}_{A}(\lambda) \int_{\R^2} E(\lambda,x,y) f(x) g(y) dx \,dy \;\varrho(d\lambda)\;.$$
	For $\varrho$-a.a.~$\lambda$, we have
	$$ E(\lambda,x,y) = \sum_{i=1}^{I(\lambda)} f^{(i)}_\lambda(x) f^{(i)}_\lambda(y)\;,$$
	where $I(\lambda) \in\{1,2\}$ and $f^{(i)}_\lambda$ are linearly independent solutions of $-y_\gl'' +  y_\gl \xi = \lambda y_\gl$.
\end{proposition}

\begin{remark}
	Once the Anderson localization of $\cH$ is established, we can deduce that $I(\lambda) = 1$ for $\varrho$-a.a.~$\lambda$.
\end{remark}

The functions $f^{(i)}_\lambda$ that appear in this statement are called \emph{generalized eigenfunctions}. The next result shows that these generalized eigenfunctions do not grow exponentially fast at infinity.
\begin{proposition}\label{Prop:GEF}
	Almost surely, for $\varrho$-almost all $\gl \in\R$ and for every $i\le I(\lambda)$
	\begin{equation}\label{Eq:Growth}
		\frac{ \ln (1+ \vert f^{(i)}_\gl(x)\vert)}{\vert x\vert} \to 0\;,\quad  x \to \pm\infty\;.
	\end{equation}
\end{proposition}

This is to be compared with the exponential growth stated in Oseledec Theorem, see Subsection \ref{Subsec:Lyapunov}. Indeed the latter asserts that, for any given $\lambda$, almost surely any given solution of $-y_\gl'' + y_\gl \xi = \lambda y_\gl$ either grows or decays exponentially fast at $+\infty$ (resp.~at $-\infty$). On the other hand, Proposition \ref{Prop:GEF} implies that almost surely, for almost all $\lambda$ w.r.t.~the spectral measure the generalized eigenfunctions do not grow exponentially fast at $\pm\infty$. At first sight, the combination of these two facts seems to imply that the generalized eigenfunctions decay exponentially fast at $\pm\infty$. However, there is a very delicate $\lambda$-dependence of the negligible sets in the first assertion. We will see in Section \ref{Sec:Kotani} how the Kotani-Simon approach takes advantage of these two facts to show that the spectrum is pure point.\\

Let us now explain the difficulty arising from the singularity of white noise in establishing the last two propositions. The delicate point consists in showing that $f^{(i)}_\lambda$ is a solution of $-y_\gl'' + y_\gl \xi = \lambda y_\gl$. Indeed, by applying standard arguments one can show the disintegration, the fact that $f^{(i)}_\gl$ belongs to $L^2_{\mbox{\tiny loc}}$ and satisfies the growth estimate of Proposition \ref{Prop:GEF}, together with the identity
\begin{equation}\label{Eq:fgl}
	\langle f^{(i)}_\gl, (\cH-\lambda)u_0 \rangle = 0\;,
\end{equation}
for a dense set of functions $u_0 \in \cD(\cH)$. Heuristically, one then deduces that $\langle (\cH-\lambda) f^{(i)}_\gl, u_0 \rangle = 0$ so that, by density, $(\cH-\lambda) f^{(i)}_\gl$ should vanish and therefore $f^{(i)}_\gl$ should be a solution of $-y_\gl'' + y_\gl \xi = \lambda y_\gl$. Unfortunately, since $\xi$ is distribution-valued, it cannot be multiplied with an arbitrary element in $L^2_{\mbox{\tiny loc}}$ and therefore one cannot swap the operator $\cH-\lambda$ from $u_0$ to $f^{(i)}_\gl$.\\
Such an issue is briefly mentioned in~\cite[Rk 1 p.153]{Kovalenko}, but we have not found any result that covers a singular situation like ours.\\

To circumvent the issue discussed above, we rely on the semigroup $e^{-t\cH}$ and establish, instead of \eqref{Eq:fgl}, the identity
$$ \langle f^{(i)}_\gl, (e^{-t\cH}-e^{-t \lambda})u_0 \rangle = 0\;.$$
Since the semigroup can be applied to a large collection of functions (thanks to the connection with the PAM),  we can deduce from such an equation that $ \langle (e^{-t\cH}-e^{-t \lambda}) f^{(i)}_\gl, u_0 \rangle = 0$ and then, by density, $(e^{-t\cH}-e^{-t \lambda}) f^{(i)}_\gl = 0$ which allows to conclude.

\bigskip

Before we proceed to the proofs of Propositions \ref{Prop:Desintegration} and \ref{Prop:GEF}, we collect a preliminary fact.

\begin{lemma}\label{Lemma:TC}
	Fix $\beta'\in (0,1]$ and set $w(x) :=  \exp((1+\vert x \vert)^{\beta'})$. For any given $t>0$, the operator $M_{w^{-1}} e^{-t \cH} M_{w^{-1}}$ is a trace class operator from $L^2$ into itself, where $M_{w^{-1}}$ is the multiplication operator by $w^{-1}$.
\end{lemma}
\begin{proof}
	Recall that from Proposition \ref{Prop:Weight} and from Sobolev Embeddings that for any $\beta \in (0,\beta')$
	$$ \sup_{x} \sup_{u_0 \in L^2} \frac{\vert u(t,x)\vert}{w_{\beta,t}(x) \|u_0\|_{L^2}} < \infty\;,$$
	and therefore
	$$ \int_x  \Big(\sup_{u_0 \in L^2} \frac{\vert u(t,x)\vert}{w(x) \|u_0\|_{L^2}}\Big)^{2} dx < \infty\;.$$
	As already mentioned in the proof of Proposition \ref{Prop:PAM}, for any given $x$, $u_0\mapsto u(t,x)$ is a bounded linear form on $L^2$. Since $u(t,x) = \int_y  e^{-t\cH}(y,x)  u_0(y) dy$ we deduce that
	$$ \sup_{u_0 \in L^2} \frac{\vert u(t,x)\vert}{\|u_0\|_{L^2}} = \| e^{-t\cH}(\cdot,x) \|_{L^2}\;,$$
	and consequently
	$$ \int_x \int_y \Big( w(x)^{-1} e^{-t\cH}(y,x) \Big)^2 dy dx< \infty\;.$$
	Denoting by $M_{w^{-1}}$ the multiplication operator by $w^{-1}$, we deduce that $M_{w^{-1}} e^{-t\cH} $ is Hilbert Schmidt, that is, $M_{w^{-1}} e^{-2t \cH} M_{w^{-1}}$ is trace class.
\end{proof}

\begin{proof}[Proof of Propositions \ref{Prop:Desintegration} and \ref{Prop:GEF}]
	
	\textit{Step 1: Construction of $\varrho$.}
	Using the previous lemma, we construct a spectral measure for $\cH$. Fix $t>0$. For any bounded Borel set $A\subset \R$, there exists $C>0$ such that $\mathbf{1}_A(y) \le C e^{-t y}$ for all $y\in \R$. The spectral theorem then yields for all $f\in L^2$ the bound
	$$ \langle w^{-1} f, E(A) w^{-1} f \rangle \le C \langle w^{-1} f, e^{-t\cH} w^{-1}f \rangle\;,$$
	which ensures that $M_{w^{-1}} E(A) M_{w^{-1}}$ is trace class.
	
	Set $\varrho(A) := \Tr(M_{w^{-1}} E(A) M_{w^{-1}})$ and note that $\varrho$ is a non-negative Radon measure. It is easy to check that $\varrho$ is a spectral measure for $\cH$, that is
	$$ \varrho(A) = 0 \Leftrightarrow E(A) = 0\;.$$
	
	\textit{Step 2: Disintegration.}
	By~\cite[V, Th 1.1]{Berezanskii}, we deduce the existence of a measurable operator-valued function $\lambda \mapsto \Psi(\lambda)$ such that\footnote{The integral on the right converges in operator norm whenever $A$ is bounded}
	$$ M_{w^{-1}} E(A) M_{w^{-1}} = \int_A \Psi(\lambda) \varrho(d\lambda)\;,$$
	and such that for $\varrho$-a.a.~$\lambda$, $\Psi(\lambda)$ is a non-negative self-adjoint operator, with a trace equal to $1$, that can be obtained as the weak limit of
	$$ \frac1{\varrho(A)} M_{w^{-1}} E(A) M_{w^{-1}}\;,$$
	on any sequence of bounded Borel sets $A \downarrow \{\lambda\}$.
	
	\textit{Step 3: On the kernel of $\Psi(\lambda)$.}
	Take $u_0 \in L^2$ with compact support (that is, $u_0$ equals $0$ a.e.~on the complement of some compact set). Such a function belongs to $H^{0}_{w_{\beta',\ell}}$ for any $\ell \in \R$. Choosing $\ell < -t-1$ and observing that $w=w_{\beta',1}$, we deduce from Proposition \ref{Prop:Weight} that $M_w(e^{-t\cH}-e^{-t\lambda})u_0 \in L^2$.
	
	We claim that for $\varrho$-a.a.~$\lambda$, $\Psi(\lambda) M_w (e^{-t\cH}-e^{-t\lambda})u_0 = 0$. Indeed, we have for any $f\in L^2$ with compact support
	\begin{align*}
		\langle \Psi(\lambda) M_w (e^{-t\cH}-e^{-t\lambda})u_0, f\rangle &= \lim_{A\downarrow \{\lambda\}} \frac1{\varrho(A)} \langle M_{w^{-1}} E(A) (e^{-t\cH}-e^{-t\lambda})u_0, f\rangle\;,
	\end{align*}
	and
	\begin{align*}
		\langle M_{w^{-1}} E(A) (e^{-t\cH}-e^{-t\lambda})u_0, f\rangle &=  \langle E(A) (e^{-t\cH}-e^{-t\lambda})u_0, M_{w^{-1}} f\rangle\\
		&= \int_A (e^{-t \mu} - e^{-t\lambda}) \langle E(d\mu)u_0, w^{-1} f\rangle\;.
	\end{align*}
	The measure $\langle E(d\mu)u_0, w^{-1} f\rangle$ is absolutely continuous w.r.t.~the measure $\varrho(d\mu)$, since the latter is a spectral measure. Moreover,
	$$ \langle E(C)u_0, w^{-1} f\rangle = \langle M_{w^{-1}} E(C)M_{w^{-1}} wu_0, f\rangle = \int_C h(\mu) \varrho(d\mu)\;,$$
	for any bounded Borel set $C$ where
	$$ h(\mu) := \frac12 \Big( \langle wu_0+f , \Psi(\mu)(wu_0 + f)\rangle -  \langle wu_0, \Psi(\mu)wu_0\rangle - \langle f , \Psi(\mu)f\rangle\Big)\;,$$
	which gives the uniform bound $\sup_{\mu \in \R} |h(\mu)| \leq 3/2$.
	
	There exists a constant $c>0$ such that $\vert e^{-t \mu} - e^{-t\lambda} \vert \le c \vert \mu - \lambda \vert$ uniformly over all $\mu$ in a given neighborhood of $\lambda$. Consequently for all $A$ close enough to $\{\lambda\}$, we find
	\begin{align*}
		\vert \langle M_{w^{-1}} E(A) (e^{-t\cH}-e^{-t\lambda})u_0, f\rangle \vert &\le c\int_A \vert \lambda - \mu \vert \vert h(\mu)\vert \varrho(d\mu)\;,
	\end{align*}
	so that
	\begin{align*}
		\langle \Psi(\lambda) M_w (e^{-t\cH}-e^{-t\lambda})u_0, f\rangle &= \lim_{A\downarrow \{\lambda\}} \frac1{\varrho(A)} \langle M_{w^{-1}} E(A) (e^{-t\cH}-e^{-t\lambda})u_0, f\rangle = 0\;.
	\end{align*}
	This concludes the proof of the claim.
	
	\textit{Step 4: Generalized eigenfunctions and growth estimate.}
	
	Let $\psi_\lambda$ be an eigenfunction of the (compact, self-adjoint and non-negative) operator $\Psi(\lambda)$, with eigenvalue say $\kappa > 0$, and set $f_\lambda = \sqrt{\kappa} w \psi_\lambda$. We have for $\varrho$-a.a.~$\lambda$
	$$ 0 = \langle \Psi(\lambda) M_w (e^{-t\cH}-e^{-t\lambda})u_0, \psi_\lambda\rangle = \sqrt{\kappa} \langle (e^{-t\cH}-e^{-t\lambda})u_0, f_\lambda\rangle\;.$$
	Thanks to the connection with the PAM, $e^{-t\cH}-e^{-t\lambda}$ can be applied to any function in $L^2$. Since it is a symmetric operator, we deduce that
	$$\sqrt{\kappa} \langle (e^{-t\cH}-e^{-t\lambda})u_0, f_\lambda\rangle = \sqrt{\kappa} \langle u_0, (e^{-t\cH}-e^{-t\lambda}) f_\lambda\rangle\;.$$
	This is enough to deduce that $(e^{-t\cH}-e^{-t\lambda})f_\lambda = 0$. We know that in the sense of distributions on $(0,\infty)\times \R$
	$$ \partial_t (e^{-t\cH} f_\lambda) = \partial_x^2(e^{-t\cH} f_\lambda) - \xi (e^{-t\cH} f_\lambda)\;.$$
	Consequently the identity $(e^{-t\cH}-e^{-t\lambda})f_\lambda = 0$ implies that
	$$ -f_\lambda'' + f_\lambda \xi = \lambda f_\lambda\;.$$
	Moreover, by Proposition \ref{Prop:Weight}, we know that $e^{-t\lambda}f_\lambda = e^{-t\cH} f_\gl$ is a continuous function that grows not faster than $w_{\beta',t}$ at infinity. We are allowed to take $\beta' < 1$ in Lemma \ref{Lemma:TC}, and thus, we deduce that $f_\gl$ satisfies \eqref{Eq:Growth}.
	
	\textit{Step 5: Kernel.}
	From the previous step and since $-y_\gl'' + y_\gl \xi= \gl y_\gl$ admits a two-dimensional space of solutions by Lemma \ref{Lemma:ODEz}, the range of $\Psi(\lambda)$ is of dimension $I(\lambda) \in\{1,2\}$. Let $\psi^{(i)}_\gl$, $1\le i \le I(\lambda)$ be independent normalized eigenfunctions of $\Psi(\lambda)$ associated to the eigenvalues $\kappa_\gl^{(i)} >0$ and set $f^{(i)}_\lambda = \sqrt{\kappa^{(i)}_\gl} w \psi^{(i)}_\lambda$. We have thus showed that for any bounded measurable functions $f,g:\R\to\R$ with compact supports
	\begin{align*}
		\langle f, E(A) g\rangle &= \langle f w, M_{w^{-1}} E(A) M_{w^{-1}} wg\rangle = \int_A \langle fw, \Psi(\lambda) wg \rangle \varrho(d\lambda)\\
		&= \int_A \sum_{i=1}^{I(\lambda)} \kappa^{(i)}_\gl\langle fw, \psi^{(i)}_\lambda\rangle \langle gw, \psi^{(i)}_\lambda\rangle \varrho(d\lambda)\\
		&= \int_A \int_{\R^2} f(x) \Big(\sum_{i=1}^{I(\lambda)} f^{(i)}_\lambda(x) f^{(i)}_\lambda(y)\Big) g(y) dx\,dy\; \varrho(d\lambda)\;.
	\end{align*}
\end{proof}

\subsection{The density of states and the support of the spectrum}

It was proven by Fukushima and Nakao~\cite{Fukushima} that the random Radon measure
$$ \frac1{L} \sum_{k\ge 1} \delta_{\lambda_{k,L}}(d\lambda)\;,$$
converges a.s.~vaguely to the deterministic measure $n(\lambda) d\lambda$ where $n$, the so-called density of states, is the derivative of the function
$$ N(\lambda) := \frac{1}{\sqrt{2\pi} \int_0^\infty \frac1{\sqrt v} \exp(-2\lambda v - \frac{v^3}{6}) dv}\;,\quad \lambda \in \R\;.$$

It is a standard result of the literature on Schr\"odinger operators that the Laplace transform of $n(\lambda)d\lambda$ satisfies for all $t\ge 0$
\begin{equation}\label{Eq:DOSPAM}
	\int_\lambda e^{-\lambda t} n(\lambda) d\lambda = \E[u^{(0)}(t,0)]\;.
\end{equation}
A proof can be found in~\cite[Chap VI.1.2]{Carmona}: although the potential considered therein is more regular, the arguments still apply; see also~\cite{Matsuda} for the case of the Anderson hamiltonian in dimension $2$.

We now determine the almost sure spectrum of $\cH$.

\begin{proposition}\label{Prop:Support}
	Almost surely, the spectrum of $\cH$ equals $\R$.
\end{proposition}
\begin{proof}
	The ergodicity of white noise implies that the spectrum of $\cH$ is almost surely given by some deterministic set. It also implies that for any $\Delta \subset \R$, the rank of the spectral projector $E(\Delta)$ is almost surely equal to either $0$ or $+\infty$, see~\cite[Lemma V.2.1]{Carmona}. We claim that for any bounded and non-empty interval $\Delta\subset\R$ this rank is non-zero with positive probability. Then necessarily it is $+\infty$ almost surely. By considering an appropriate countable collection of sets $\Delta$, this suffices to deduce that almost surely any non-trivial interval intersects the spectrum, so that the latter is dense in $\R$.\\
	Fix a bounded Borel set $\Delta \subset \R$ and some non-negative function $f\in L^2(\R)$ of unit $L^2$ norm. We are going to show that
	\begin{equation}\label{Eq:DeltaMf}
		\int_\Delta n(E) dE = \E[ \Tr( M_f E(\Delta) M_f)]\;,
	\end{equation}
	where $M_f$ is the multiplication operator by $f$. This is enough to deduce the claim because, as soon as $\Delta$ is non-empty, the l.h.s.~is strictly positive (the density of states is positive everywhere as it can be checked from its explicit expression) and therefore the probability that $E(\Delta)$ is non-zero needs to be positive.
	
	To prove \eqref{Eq:DeltaMf}, we will check that the Laplace transforms of the two measures coincide. For any $t\ge 0$, by Fubini's Theorem and the linearity of the trace
	\begin{align*}
		\int_\lambda e^{-\lambda t}  \E[ \Tr( M_f E(d\lambda) M_f)] = \E[\Tr(M_f \int_\lambda e^{-\lambda t} E(d\lambda) M_f)] = \E[\Tr(M_f e^{-t\cH} M_f)]\;.
	\end{align*} 
	The operator $M_f e^{-t\cH} M_f$ is positive and self-adjoint. Its trace coincides with the Hilbert-Schmidt norm of its square root $M_{f} e^{-\frac{t}{2} \cH}$, which is given by
	$$ \int_{x,y} \Big({f}(x) e^{-\frac{t}{2} \cH}(x,y) \Big)^2 dx\, dy\;.$$
	Proposition \ref{Prop:PAM} and the semigroup property show that this last quantity equals
	$$ \int_x f(x) e^{-t \cH}(x,x) f(x) dx = \int_x u^{(x)}(t,x) f(x)^2 dx\;.$$
	Taking expectations, using the invariance in law of $u^{(x)}(t,x)$ under translating $x$ and the fact that $f^2$ integrates to $1$ we obtain
	$$ \E[\Tr(M_f e^{-tH} M_f)] = \E[u^{(0)}(t,0)]\;.$$
	We conclude with \eqref{Eq:DOSPAM}.
\end{proof}

\begin{remark}[Perturbation of the noise]
	The results on the PAM presented in Subsection \ref{Subsec:PAM} apply verbatim to $\xi_\varepsilon$: the only input therein is the local regularity of the noise and its growth at infinity. The growth estimates are then automatically satisfied since the argument are generic given the growth estimates on the PAM.
	The existence of the integrated density of states in the case of $\xi_\varepsilon$ remains true by ergodicity, however probably no explicit formula can be given and the positivity of the density of states cannot be ensured. However, there is an alternative way of showing that the spectrum of the Hamiltonian associated with $\xi_\varepsilon$ is dense in $\R$ using Weyl sequences. Recall that, given a self-adjoint operator $A$, $E$ lies in the spectrum of $A$ if there exists a Weyl sequence at energy $E$, that is, if there exists a sequence $\psi_n$ such that $\| \psi_n \| = 1$ and $\| (A - E)\psi_n\| \to 0$ as $n\to\infty$. See for instance~\cite[Lemma 2.17]{Teschl}.\\
	Fix $E\in\R$. For any $n\ge 1$, consider the shift invariant event
	$$ A_n := \Big\{\exists x \in \R: |\xi_\varepsilon(y) - E| \le 1/n\;, \forall y \in (x-n,x+n)\Big\}\;.$$
	It is not hard to check that the probability of this event is strictly positive, and, by ergodicity of $\xi_\varepsilon$, it is equal to $1$. Let $\chi_n$ be a smooth indicator function of the interval $(-n,n)$. Working on $\cap_n A_n$, we can find a random sequence $x_n$ such that the function
	$$ x\mapsto \sin(\frac{1}{\sqrt n} x) \chi_{n}(x-x_n)\;,$$
	up to an $L^2$ normalization, is a Weyl sequence at energy $E$ of $\cH^{(\varepsilon)}$.
\end{remark}

\section{Kotani-Simon approach}\label{Sec:Kotani}

In this section, we present the main lines of an alternative proof of Theorem \ref{Th:Expo} which is based on a famous article of Kotani and Simon~\cite{KotaniSimon}. Let $\cF_{\R\backslash [0,t_0]}$ be the sigma-field generated by $\xi$ outside $[0,t_0]$ for some $t_0>0$. Let $\sigma$ be an arbitrary spectral measure.\\

The main ingredient that is required is the following. One needs to know that conditionally given $\cF_{\R\backslash [0,t_0]}$, the law of $\theta_\gl(t_0)$ starting from $\theta_\gl(0) = \theta_0 \in [0,\pi)$ has a density which is bounded uniformly over all $\theta_0 \in [0,\pi)$ and uniformly over all $\lambda$ in compact sets. In the white noise case considered in this article, this bound is trivially guarantied by the independence properties of the white noise and the bound \eqref{Eq:plambda} on the density of the diffusion.\\

Given this bound, the exact same arguments as in~\cite[Sec. 2]{KotaniSimon} show that the spectral measure averaged over the randomness coming from the restriction of $\xi$ to $[0,t_0]$
$$ w(d\lambda) := \E\big[\sigma(d\lambda) \, \mid \, \cF_{\R\backslash [0,t_0]}\big]\;,$$
is absolutely continuous.\\

We also need some results that derive from Furstenberg and Oseledec Theorems, as presented at the end of Subsection \ref{Subsec:Lyapunov}. For any given $\gl\in\R$, let $A_\gl^\pm$ be the following events: there exists an angle $\alpha_\gl^\pm \in [0,\pi)$ such that	
	$$ \frac{\ln r_{\gl,\alpha_\gl^\pm}(t)}{\vert t \vert } \to -\gamma_\gl\;,\quad t\to \pm \infty\;,$$
	and for all $\theta_0 \ne \alpha_\gl^\pm$ we have
	$$ \frac{\ln r_{\gl,\theta_0}(t)}{\vert t \vert} \to \gamma_\gl\;,\quad t\to \pm \infty\;.$$
For $\gl$ fixed, the events $A^+_\gl$ and $A^-_\gl$ have probability one. By Fubini's Theorem, there exists a random set $\cL \subset \R$, whose complement is of zero Lebesgue measure and such that almost surely, for all $\lambda \in \cL$, the events $A_\gl^+$ and $A_\gl^-$ are satisfied.\\
An important remark should be made:  the events $A_\gl^\pm$ belong to $\cF_{\R\backslash [0,t_0]}$. Indeed, since Oseledec Theorem is a deterministic consequence of the convergence \eqref{Eq:Furstenberg} implied by Furstenberg Theorem, the events $A_\gl^\pm$ coincide, up to a negligible set, with the event \eqref{Eq:Furstenberg}. The latter only depends on the behavior of $\xi$ at $\pm\infty$, consequently it is measurable w.r.t.~$\cF_{\R\backslash [0,t_0]}$, and the same holds for $A_\gl^\pm$. (Of course the r.v.~$\alpha_\gl^+$ depends on the behavior of $\xi$ on $[0,t_0]$, but its existence only requires knowledge on the tail sigma field.)

With these ingredients at hand, we can proceed to the proof of Anderson localization.

\begin{proof}[Alternative proof of Theorem \ref{Th:Expo}]
By absolute continuity almost surely
$$ w(\R\backslash \cL) = 0\;.$$
Since $\cL$ is $\cF_{\R\backslash [0,t_0]}$-measurable we deduce that almost surely $\sigma(\R\backslash \cL) = 0$. Almost surely for every $\gl\in\cL$, every solution of the eigenproblem $-y_\gl'' + y_\gl \xi = \lambda y_\gl$ either decays exponentially at both infinities (and is therefore in $L^2$) or its modulus (that is, $\sqrt{y_\gl^2 + (y'_\gl)^2}$ ) grows exponentially at least at one infinity. Denote by $D$ the random set of all $\gl \in \cL$ such that there exists a solution of the eigenproblem that decays exponentially at both infinities: this set lies in the point spectrum of $\cH$ and is therefore necessarily countable. Almost surely for every $\gl\in\cL\backslash D$, the moduli of all solutions of the eigenproblem grow exponentially at least at one infinity and therefore by Proposition \ref{Prop:GEF}, almost surely $\sigma(\cL\backslash D) = 0$. Consequently almost surely $\sigma$ only charges the countable set $D$ and is therefore pure point. In addition, we deduce that the eigenfunctions satisfy the exponential decay estimate of the statement of the theorem. Regarding the density of the spectrum, it was already proven in Proposition \ref{Prop:Support}.
\end{proof}

\begin{remark}[Perturbation of the noise]
	The present proof is quite robust, compared to the one presented in the previous section. Consider the noise $\xi_\varepsilon$. The absolute continuity of $\theta_\gl(t_0)$ given $\cF_{\R\backslash [0,t_0]}$ can be derived by a perturbation argument, similar to~\cite[Sec 3]{KotaniSimon}. Regarding the exponential growth/decay of the solution to the ODE, the proof proceeds as follows. The convergence of $\frac1{n} \ln \| U_\gl(n)\|$ can be derived from the sub-additive Kingman ergodic theorem, and its extension to real parameters $t>0$ can be derived similarly as we did in Subsection \ref{Subsec:Lyapunov}. The strict positivity of the limit (for almost all $\lambda \in \R$) is more subtle and can be deduced from a deep result of Kotani~\cite[Th 4.3]{Kotani}. Then one can apply Oseledec Theorem as in the white noise case.
\end{remark}

\bibliographystyle{Martin}
\bibliography{library}

\begin{thebibliography}{{Mat}20}
\expandafter\ifx\csname url\endcsname\relax
  \def\url#1{\texttt{#1}}\fi
\expandafter\ifx\csname urlprefix\endcsname\relax\def\urlprefix{URL }\fi
\expandafter\ifx\csname href\endcsname\relax
  \def\href#1#2{#2}\fi
\expandafter\ifx\csname burlalt\endcsname\relax
  \def\burlalt#1#2{\href{#2}{\texttt{#1}}}\fi

\bibitem[AC15]{AllezChouk}
\textsc{R.~{Allez}} and \textsc{K.~{Chouk}}.
\newblock {The continuous Anderson hamiltonian in dimension two}.
\newblock \emph{ArXiv e-prints} (2015).
\newblock \burlalt{arXiv:1511.02718}{http://arxiv.org/abs/1511.02718}.

\bibitem[AM93]{AM}
\textsc{M.~Aizenman} and \textsc{S.~Molchanov}.
\newblock Localization at large disorder and at extreme energies: an elementary
  derivation.
\newblock \emph{Comm. Math. Phys.} \textbf{157}, no.~2, (1993), 245--278.

\bibitem[And58]{Anderson58}
\textsc{P.~W. Anderson}.
\newblock {Absence of Diffusion in Certain Random Lattices}.
\newblock \emph{Physical Review} \textbf{109}, no.~5, (1958), 1492--1505.
\newblock
  \burlalt{doi:10.1103/PhysRev.109.1492}{http://dx.doi.org/10.1103/PhysRev.109.1492}.

\bibitem[AW15]{AizenmanWarzel}
\textsc{M.~Aizenman} and \textsc{S.~Warzel}.
\newblock \emph{Random operators}, vol. 168 of \emph{Graduate Studies in
  Mathematics}.
\newblock American Mathematical Society, Providence, RI, 2015.
\newblock Disorder effects on quantum spectra and dynamics.

\bibitem[BDM22]{BDM}
\textsc{I.~{Bailleul}}, \textsc{N.~V. {Dang}}, and \textsc{A.~{Mouzard}}.
\newblock {Analysis of the Anderson operator}.
\newblock \emph{arXiv e-prints}  arXiv:2201.04705.
\newblock \burlalt{arXiv:2201.04705}{http://arxiv.org/abs/2201.04705}.

\bibitem[Ber68]{Berezanskii}
\textsc{J.~M. Berezanskii}.
\newblock \emph{Expansions in eigenfunctions of selfadjoint operators}.
\newblock Translations of Mathematical Monographs, Vol. 17. American
  Mathematical Society, Providence, R.I., 1968.
\newblock Translated from the Russian by R. Bolstein, J. M. Danskin, J. Rovnyak
  and L. Shulman.

\bibitem[BL85]{BougerolLacroix}
\textsc{P.~Bougerol} and \textsc{J.~Lacroix}.
\newblock \emph{Products of random matrices with applications to
  {Schr{\"o}dinger} operators}, vol.~8 of \emph{Prog. Probab. Stat.}
\newblock Birkh{\"a}user, Boston, MA, 1985.

\bibitem[Car82]{CarmonaDuke}
\textsc{R.~Carmona}.
\newblock Exponential localization in one dimensional disordered systems.
\newblock \emph{Duke Math. J.} \textbf{49}, (1982), 191--213.
\newblock
  \burlalt{doi:10.1215/S0012-7094-82-04913-4}{http://dx.doi.org/10.1215/S0012-7094-82-04913-4}.

\bibitem[CL90]{Carmona}
\textsc{R.~Carmona} and \textsc{J.~Lacroix}.
\newblock \emph{Spectral theory of random {S}chr\"{o}dinger operators}.
\newblock Probability and its Applications. Birkh\"{a}user Boston, Inc.,
  Boston, MA, 1990.

\bibitem[CvZ21]{CvZ}
\textsc{K.~Chouk} and \textsc{W.~van Zuijlen}.
\newblock Asymptotics of the eigenvalues of the {A}nderson {H}amiltonian with
  white noise potential in two dimensions.
\newblock \emph{Ann. Probab.} \textbf{49}, no.~4, (2021), 1917--1964.
\newblock
  \burlalt{doi:10.1214/20-aop1497}{http://dx.doi.org/10.1214/20-aop1497}.

\bibitem[DL20]{DL_Bottom}
\textsc{L.~Dumaz} and \textsc{C.~Labb\'{e}}.
\newblock Localization of the continuous {A}nderson {H}amiltonian in 1-{D}.
\newblock \emph{Probab. Theory Related Fields} \textbf{176}, no. 1-2, (2020),
  353--419.
\newblock
  \burlalt{doi:10.1007/s00440-019-00920-6}{http://dx.doi.org/10.1007/s00440-019-00920-6}.

\bibitem[DL21a]{DL_Crossover}
\textsc{L.~{Dumaz}} and \textsc{C.~{Labb{\'e}}}.
\newblock {Localization crossover for the continuous Anderson Hamiltonian in
  $1$-d}.
\newblock \emph{arXiv e-prints}  arXiv:2102.09316.
\newblock \burlalt{arXiv:2102.09316}{http://arxiv.org/abs/2102.09316}.

\bibitem[DL21b]{DL_Critical}
\textsc{L.~{Dumaz}} and \textsc{C.~{Labb{\'e}}}.
\newblock {The delocalized phase of the Anderson Hamiltonian in $1$-d}.
\newblock \emph{Ann. Probab. to appear}  arXiv:2102.05393.
\newblock \burlalt{arXiv:2102.05393}{http://arxiv.org/abs/2102.05393}.

\bibitem[FN77]{Fukushima}
\textsc{M.~Fukushima} and \textsc{S.~Nakao}.
\newblock On spectra of the {S}chr\"odinger operator with a white {G}aussian
  noise potential.
\newblock \emph{Z. Wahrscheinlichkeitstheorie und Verw. Gebiete} \textbf{37},
  no.~3, (1976/77), 267--274.

\bibitem[FS83]{FS}
\textsc{J.~Fr\"{o}hlich} and \textsc{T.~Spencer}.
\newblock Absence of diffusion in the {A}nderson tight binding model for large
  disorder or low energy.
\newblock \emph{Comm. Math. Phys.} \textbf{88}, no.~2, (1983), 151--184.

\bibitem[GIP15]{Max}
\textsc{M.~Gubinelli}, \textsc{P.~Imkeller}, and \textsc{N.~Perkowski}.
\newblock Paracontrolled distributions and singular {PDE}s.
\newblock \emph{Forum Math. Pi} \textbf{3}, (2015), e6, 75.
\newblock \burlalt{arXiv:1210.2684}{http://arxiv.org/abs/1210.2684}.
\newblock
  \burlalt{doi:10.1017/fmp.2015.2}{http://dx.doi.org/10.1017/fmp.2015.2}.

\bibitem[GMP77]{GMP}
\textsc{I.~J. Goldsheid}, \textsc{S.~A. Molcanov}, and \textsc{L.~A. Pastur}.
\newblock A random homogeneous {S}chr\"odinger operator has a pure point
  spectrum.
\newblock \emph{Funkcional. Anal. i Prilo\v zen.} \textbf{11}, no.~1, (1977),
  1--10, 96.

\bibitem[GUZ20]{GUZ}
\textsc{M.~Gubinelli}, \textsc{B.~Ugurcan}, and \textsc{I.~Zachhuber}.
\newblock Semilinear evolution equations for the {A}nderson {H}amiltonian in
  two and three dimensions.
\newblock \emph{Stoch. Partial Differ. Equ. Anal. Comput.} \textbf{8}, no.~1,
  (2020), 82--149.
\newblock
  \burlalt{doi:10.1007/s40072-019-00143-9}{http://dx.doi.org/10.1007/s40072-019-00143-9}.

\bibitem[Hai14]{Hairer2014}
\textsc{M.~Hairer}.
\newblock A theory of regularity structures.
\newblock \emph{Invent. Math.} \textbf{198}, no.~2, (2014), 269--504.
\newblock \burlalt{arXiv:1303.5113}{http://arxiv.org/abs/1303.5113}.
\newblock
  \burlalt{doi:10.1007/s00222-014-0505-4}{http://dx.doi.org/10.1007/s00222-014-0505-4}.

\bibitem[HL15]{PAM2D}
\textsc{M.~Hairer} and \textsc{C.~Labbé}.
\newblock {A simple construction of the continuum parabolic Anderson model on
  $\mathbf{R}^2$}.
\newblock \emph{Electronic Communications in Probability} \textbf{20}, (2015),
  1 -- 11.
\newblock
  \burlalt{doi:10.1214/ECP.v20-4038}{http://dx.doi.org/10.1214/ECP.v20-4038}.

\bibitem[Kir08]{Kirsch}
\textsc{W.~Kirsch}.
\newblock An invitation to random {S}chr\"odinger operators.
\newblock In \emph{Random {S}chr\"odinger operators}, vol.~25 of \emph{Panor.
  Synth\`eses},  1--119. Soc. Math. France, Paris, 2008.
\newblock With an appendix by Fr\'ed\'eric Klopp.

\bibitem[Kot84]{Kotani}
\textsc{S.~Kotani}.
\newblock Ljapunov indices determine absolutely continuous spectra of
  stationary random one-dimensional {Schr{\"o}dinger} operators.
\newblock Stochastic analysis, {Proc}. {Taniguchi} {Int}. {Symp}., {Katata} \&
  {Kyoto}/{Jap}. 1982, {North}-{Holland} {Math}. {Libr}. 32, 225-247 (1984).,
  1984.

\bibitem[KS78]{Kovalenko}
\textsc{V.~F. Kovalenko} and \textsc{J.~A. Semenov}.
\newblock Some questions on expansions in generalized eigenfunctions of a
  {S}chr\"{o}dinger operator with strongly singular potentials.
\newblock \emph{Uspekhi Mat. Nauk} \textbf{33}, no. 4(202), (1978), 107--140,
  255.

\bibitem[KS87]{KotaniSimon}
\textsc{S.~Kotani} and \textsc{B.~Simon}.
\newblock Localization in general one-dimensional random systems. {II}:
  {Continuum} {Schr{\"o}dinger} operators.
\newblock \emph{Commun. Math. Phys.} \textbf{112}, (1987), 103--119.
\newblock
  \burlalt{doi:10.1007/BF01217682}{http://dx.doi.org/10.1007/BF01217682}.

\bibitem[KS81]{kunz1980}
\textsc{H.~Kunz} and \textsc{B.~Souillard}.
\newblock Sur le spectre des op\'erateurs aux diff\'erences finies
  al\'eatoires.
\newblock \emph{Comm. Math. Phys.} \textbf{78}, no.~2, (1980/81), 201--246.

\bibitem[Lab19]{Lab19}
\textsc{C.~Labb{\'e}}.
\newblock {The continuous Anderson hamiltonian in $d\le 3$}.
\newblock \emph{{Journal of Functional Analysis}} \textbf{277}, no.~9(2019).
\newblock
  \burlalt{doi:10.1016/j.jfa.2019.05.027}{http://dx.doi.org/10.1016/j.jfa.2019.05.027}.

\bibitem[LS75]{Levitan}
\textsc{B.~M. Levitan} and \textsc{I.~S. Sargsjan}.
\newblock \emph{Introduction to spectral theory: selfadjoint ordinary
  differential operators}.
\newblock Translations of Mathematical Monographs, Vol. 39. American
  Mathematical Society, Providence, R.I., 1975.
\newblock Translated from the Russian by Amiel Feinstein.

\bibitem[{Mat}20]{Matsuda}
\textsc{T.~{Matsuda}}.
\newblock {Integrated density of states of the Anderson Hamiltonian with
  two-dimensional white noise}.
\newblock \emph{arXiv e-prints}  arXiv:2011.09180.
\newblock \burlalt{arXiv:2011.09180}{http://arxiv.org/abs/2011.09180}.

\bibitem[Mou22]{Mouzard}
\textsc{A.~Mouzard}.
\newblock Weyl law for the {A}nderson {H}amiltonian on a two-dimensional
  manifold.
\newblock \emph{Ann. Inst. Henri Poincar\'{e} Probab. Stat.} \textbf{58},
  no.~3, (2022), 1385--1425.
\newblock
  \burlalt{doi:10.1214/21-aihp1216}{http://dx.doi.org/10.1214/21-aihp1216}.

\bibitem[Tes14]{Teschl}
\textsc{G.~Teschl}.
\newblock \emph{Mathematical methods in quantum mechanics}, vol. 157 of
  \emph{Graduate Studies in Mathematics}.
\newblock American Mathematical Society, Providence, RI, second ed., 2014.
\newblock With applications to Schr\"odinger operators.

\bibitem[Wei87]{Weidmann}
\textsc{J.~Weidmann}.
\newblock \emph{Spectral theory of ordinary differential operators}, vol. 1258
  of \emph{Lecture Notes in Mathematics}.
\newblock Springer-Verlag, Berlin, 1987.

\end{thebibliography}

\end{document}